\documentclass[11pt,a4paper]{article}
\usepackage[english]{babel}
\usepackage{amsthm}
\usepackage{hyperref}
\usepackage{mathtools}
\usepackage{amsfonts}
\usepackage{amssymb}
\usepackage{esint}
\usepackage{xcolor}
\usepackage[shortlabels]{enumitem}
\usepackage{bigints}
\usepackage{dsfont}
\usepackage{mathrsfs}
\usepackage{ragged2e}

\usepackage{array} %to make tables
\usepackage[table]{xcolor}
\arrayrulecolor{blue}

\usepackage{quiver}

\newtheorem{thm}{Theorem}[subsection]
\newtheorem*{main}{Quantum Ergodicity Theorem}
\newtheorem{QELS}{Theorem}

\newtheorem{propx}{Proposition}[subsection]

\newtheorem{lemma}{Lemma}[subsection]
\newtheorem{claim}{Claim}[subsection]

\theoremstyle{definition}
\newtheorem{defn}{Definition}[subsection]
\newtheorem{ex}{Example}[subsection]
\newtheorem*{pc}{Proof of Claim}

\theoremstyle{remark}
\newtheorem{rem}{Remark}[subsection]

\usepackage{xassoccnt}  %serve per i comandi qui sotto
\DeclareCoupledCountersGroup{theorems}  %numerazione progressiva di Thm, prop, etc...
\DeclareCoupledCounters[name=theorems]{thm,defn,lemma, rem, ex, propx, claim, cor}

\newcommand{\R}{\mathbb{R}}
\newcommand{\C}{\mathbb{C}}
\newcommand{\N}{\mathbb{N}}

\newcommand{\Hyp}{\mathbb{H}}

\newcommand{\eps}{\varepsilon}
\newcommand{\la}{\langle}
\newcommand{\ra}{\rangle}

\newcommand{\1}{\mathds{1}}

\newcommand{\cu}{\mathcal{C}}

\newcommand{\PSL}{\operatorname{PSL}}

\newcommand{\F}{\mathcal{F}}
\newcommand{\Pl}{\operatorname{Pl}}

\newcommand{\vol}{\operatorname{vol}}
\newcommand{\hyp}{\operatorname{hyp}}
\newcommand{\para}{\operatorname{par}}

\usepackage[OT1]{fontenc}
\DeclareFontFamily{OT1}{pzc}{}
\DeclareFontShape{OT1}{pzc}{m}{it}{<-> s * [1.35] pzcmi7t}{}
\DeclareMathAlphabet{\mathcal}{OT1}{pzc}{m}{it}
\newcommand {\ca}{\mathcal{a}}

\title{Spectral Convergence of Hyperbolic Surfaces and Quantum Ergodicity}
\author{Giacomo Gavelli}
\date{}

\begin{document}

\maketitle

\begin{abstract}
We extend the notion of Plancherel convergence, introduced in the study of compact quotients of locally compact groups, to finite-area hyperbolic surfaces and establish a quantum ergodicity theorem for Plancherel sequences. We show that the expectation values of integral operators whose kernels are uniformly bounded and have uniformly bounded propagation become asymptotically equidistributed, on average, for eigenvalues in any fixed compact interval in $(\frac{1}{4},\infty).$ Our result recovers the case of multiplication operators studied by Le Masson and Sahlsten, corresponding to the degenerate case of propagation bound zero and kernels supported on the diagonal. Moreover, our result allows the systole to shrink without a uniform lower bound, with the allowed degeneration constrained by Plancherel convergence. This identifies Plancherel convergence as a natural spectral framework for quantum ergodicity on degenerating hyperbolic surfaces.
\end{abstract}

{\small\tableofcontents}

\section*{Introduction}

Let $M$ be a closed Riemannian manifold with ergodic geodesic flow (a sufficient condition for the geodesic flow being ergodic is that the manifold has strictly negative curvature). Let $0=\lambda_0<\lambda_1\le\dots$ be the non-decreasing sequence of Laplace eigenvalues counted with multiplicities and $\{\psi_j\}_{j\in\N}$ be an associated orthonormal basis for $L^2(M)$. The Quantum Ergodicity (QE) Theorem, which is a culmination of results of Shnirelman \cite{sni1974erg}, Zelditch \cite{zelditch1987uniform} and Colin de Verdière \cite{colin1985ergodicite}, states that for any pseudodifferential operator $A$ of order zero with principal symbol $\sigma_A\in C^\infty(S^*M)$
\begin{equation*}\label{original QE}
    \lim_{\lambda\to\infty}\frac{1}{\Big|\{j\in\N\ \colon \lambda_j\le\lambda\}\Big|}\sum_{j\ \colon \lambda_j\le\lambda}\Bigg|\la A\psi_j,\psi_j\ra-\fint_{S^*M}\sigma_A dw\Bigg|^2=0,
\end{equation*}
where $dw$ is the Liouville measure on the unit tangent bundle and whenever $(X,\mu)$ is a measure space of finite measure we write $\fint _X f(x)d\mu(x):=\frac{1}{\mu (X)}\int_X f(x)d\mu(x)$ for any $f\in L^1(X)$. 
From a different perspective, one can ask whether eigenfunctions with eigenvalues in a fixed spectral window also equidistribute under suitable \textit{large-scale limits}. This setting is referred to as the \textit{level aspect}. A common large-scale limit to consider when studying quantum chaos on hyperbolic surfaces is the Benjamini--Schramm limit, which we now recall. If $X=\Gamma\setminus\Hyp$ is a finite-area hyperbolic surface and $p\in X$, the \textit{injectivity radius} of $X$ at $p$ is 
$$\operatorname{inj}_X(p)=\frac{1}{2}\inf\{d(z,\gamma. z)\ :\ \gamma\in\Gamma\smallsetminus\{\text{id}\}\},$$
where $z$ is any lift of $p$ to $\Hyp$. For $R>0$, the $R$-\textit{thin part} of $X$ is 
$$X_{<R}=\{ p\in X\ :\ \operatorname{inj}_X(p)<R\}.$$

We say that a sequence $(X_n)_{n\in\N}$ of finite-area hyperbolic surfaces is \textbf{Benjamini--Schramm convergent} to $\Hyp$ if for any $R>0$
$$\frac{\operatorname{vol}((X_n)_{<R})}{\operatorname{vol}(X_n)}\xrightarrow[]{n\to\infty}0.$$

We further recall that the \textit{systole} of a hyperbolic surface $X$ is the length of a shortest closed geodesic on $X$. In \cite{le2017quantum}, Le Masson and Sahlsten showed that for a fixed compact interval $I\subset(1/4,\infty)$, a uniformly discrete Benjamini--Schramm sequence $(X_n)_{n\in\N}$ of closed hyperbolic surfaces with a uniform spectral gap and a uniformly bounded sequence of measurable functions $(a_n)_{n\in\N}$ the following Quantum Ergodicity result holds:
\begin{equation*}\label{Le-Sahl QE}
    \lim_{n\to\infty} \frac{1}{N(X_n,I)}\sum_{j\ \colon\lambda_n\in I}\Big|\la a_n\psi_j^{(n)},\psi_j^{(n)}\ra -\fint_{X_n} a_nd\vol_{X_n}\Big|^2=0,
\end{equation*}
where $N(X_n,I)$ is the number of eigenvalues of $X_n$ in $I$ counted with multiplicity. This result was generalized in \cite{abert2022eigenfunctions} to all rank 1 locally symmetric spaces, allowing operators by convolution with invariant kernels rather than multiplication operators. An extension to higher rank locally symmetric spaces is given in \cite{brumley2026quantum}. In the large eigenvalue regime, Zelditch proved that a Quantum Ergodicity Theorem holds for noncompact hyperbolic surfaces with finite area\cite{zelditch1991mean}. Le Masson and Sahlsten then provided the level aspect counterpart \cite{le2024quantum}. Extending the latter result is the main focus of this paper, so we recall it here. 

A non-compact hyperbolic surface $X$ of finite area has both discrete and continuous spectrum for the Laplacian. If $0=\lambda_0<\lambda_1\le\dots$ is the nondecreasing collection of discrete eigenvalues, we can fix an orthonormal system $\{\psi_j\}$ of $L^2$-eigenfunctions of the Laplacian. The continuous spectrum is the whole interval $[\frac{1}{4},\infty)$. Let $\mathcal{C}(X):=\{\ca_1,\dots,\ca_q\}$ be the set of cusps of $X$. For every $s\in\C$ and $k\in\{1,\dots,q\}$ there is a certain non-$L^2$ eigenfunction of the Laplacian $E_k(\cdot,s)$, called Eisenstein series (we refer to \cite{iwaniec2021spectral} for a thorough treatment or to sections \ref{cusps}, \ref{Eisenstein series} for some background). For a fixed interval $I\subset(\frac{1}{4},\infty)$ we define 
$$M(X,I):=\frac{1}{4\pi}\int_{\tau^{-1}(I)} \frac{-\varphi'_X}{\varphi_X}\Big(\frac{1}{2}+is\Big)ds,$$
where $\tau(s):=\frac{1}{4}+s^2$ and $\varphi_X$ is the scattering determinant of $X$ (see section \ref{Eisenstein series} or \cite{iwaniec2021spectral}). As for the case of compact surfaces, we denote by $N(X,I)$ the number of discrete eigenvalues in the interval $I$ counted with multiplicities. The quantity $N(X,I)+M(X,I)$ measures the contribution of the discrete and continuous spectra of $X$ in the interval $I$. The \textit{Quantum mean absolute deviation} of $X$ over $I$ with respect to a function $a\in L^\infty(X)$ is 
\begin{align*}
    &\operatorname{Dev}_{X,I}(a):=\frac{1}{N(X,I)+M(X,I)}\Bigg(\sum_{j\ \colon \lambda_j\in I}\big|\la a\psi_j,\psi_j\ra-\la a\ra\big|+ \\
    &\frac{1}{4\pi}\int_{\tau^{-1}(I)}\Big|\sum_{k=1}^q \la a E_k(\cdot,\frac{1}{2}+is),E_k(\cdot,\frac{1}{2}+is)\ra +\frac{\varphi_X'}{\varphi_X}\bigg(\frac{1}{2}+is\bigg)\la a \ra\Big|ds\Bigg),
\end{align*}
where $\la a\ra=\fint _X a(x)d\vol(x)$. Le Masson and Sahlsten prove the following result. 

\begin{QELS}[\cite{le2024quantum}, Theorem 1.2]\label{thm to extend}
    Let  $I\subset(\frac{1}{4},\infty)$ be a fixed compact interval and $(X_n)_{n\in\N}$ be a sequence of finite-area hyperbolic surfaces such that: 
\begin{enumerate}[1)]
    \item $X_n$ is Benjamini--Schramm convergent to $\Hyp$ as $n\to\infty$;
    \item there is a uniform lower bound on the systole of $X_n$;
    \item there is a uniform spectral gap for the Laplacian on $X_n$;
    \item the number of cusps $q_n$ of $X_n$ is such that $\frac{q_n^2}{\vol (X_n)}\xrightarrow[]{n\to\infty}0$.
\end{enumerate}
Then for any sequence of uniformly bounded functions $a_n\in L^\infty(X_n)$ with uniformly compact support \footnote{For $\upsilon\ge1$, any hyperbolic surface can be decomposed into a disjoint union of cuspidal zones and a compact part $X(\upsilon)$ with boundary consisting of horocycles (see section \ref{cusps}). If $(X_n)_{n\in\N}$ is a sequence of hyperbolic surfaces we say that functions $a_n\in L^\infty (X_n)$ have \textit{uniformly compact support} if there exists $\upsilon>1$ such that $\operatorname{supp} a_n\subseteq X_n(\upsilon)$ for every $n\in\N$.} we have 
$$\operatorname{Dev}_{X_n,I}(a_n)\xrightarrow[]{n\to\infty}0.$$
\end{QELS}

We aim to extend this result in two directions:

1) In the spirit of \cite{abert2022eigenfunctions}, we show that a Quantum Ergodicity result holds when one considers invariant integral operators with well behaved kernels instead of multiplication operators. That is, we show that if $(X_n)_{n\in\N}$ is a sequence satisfying conditions $1)-4)$ in Theorem \ref{thm to extend}, then for any sequence of uniformly bounded kernels $k_n:X_n\times X_n\to\C$ with uniformly relative compact support (see Def. \ref{uniform relative compact support}) and uniformly finite propagation
$$\operatorname{Dev}_{X_n,I}(T_n)\xrightarrow[]{n\to\infty}0,$$
where $T_n$ denotes the operator by convolution with the kernel $k_n$. In fact, we show that this result holds with conditions less restrictive than $1)-4)$. This is the second extension.

2) The proof of Theorem \ref{thm to extend} relies, among other things, on the fact that if $(X_n)_{n\in\N}$ is a sequence of finite-area hyperbolic surfaces satisfying $1)+2)$, then 
\begin{equation}\label{plancherel measure of interval}
\lim_{n\to\infty}\frac{N(X_n,I)+M(X_n,I)}{\vol(X_n)}=\frac{1}{4\pi}\int_\R \1_I\big(\frac{1}{4}+s^2\Big)\tanh(\pi s)sds,
\end{equation}
where we call the RHS of the above expression the \textit{Plancherel measure of the interval} $I$ and we denote it by $\mu_{Pl}(I)$. We call a sequence satisfying (a slight variation of) the above asymptotic a \textbf{Plancherel sequence} (see section \ref{spectral convergence}) and we show that the Plancherel condition already suffices to prove the Quantum Ergodicity Theorem.

The relation between Plancherel and Benjamini--Schramm sequences in the case of compact surfaces has been extensively studied (see e.g. \cite{abert2017growth}, \cite{deitmar2019benjamini}, \cite{gavelli2025benjamini}): if $(X_n)_{n\in\N}$ is a Plancherel sequence then $X_n$ is Benjamini--Schramm convergent to $\Hyp$ and if $(X_n)_{n\in\N}$ is a Benjamini--Schramm sequence with a uniform lower bound on the systole, then it is a Plancherel sequence. Moreover, all these implications are strict. We show that the same relations hold for surfaces with cusps as well (Proposition \ref{PL vs BS}). 

The main result of the paper is the following Quantum Ergodicity Theorem. 

\begin{main}
    Let $I\subset (1/4,\infty)$ be a compact interval and $(X_n)_{n\in\N}$ a sequence of finite area hyperbolic surfaces. Denote by $q_n$ the number of cusps of $X_n$. Assume that 
    \begin{enumerate}
        \item $(X_n)_{n\in\N}$ is a Plancherel sequence;
        \item there is a uniform spectral gap for the Laplacian on $X_n$;
        \item $\frac{q_n^2}{\vol (X_n)}\xrightarrow[]{n\to\infty}0$.
    \end{enumerate}
        
     Then, for any  sequence of uniformly bounded  kernels $(k_n)_{n\in\N}$ with uniformly relative compact support (see Def. \ref{uniform relative compact support}) and uniformly finite propagation,

    $$\operatorname{Dev}_{X_n,I}(T_n)\xrightarrow[]{n\to\infty}0,$$
    where $T_n$ denotes the operator by convolution with the kernel $k_n$.
\end{main}

It is worth noticing that the authors of \cite{le2024quantum} already state that their techniques allow the systole to shrink to $0$, as long as that does not happen ``too fast" compared to the volume growth. The main technical point of this paper is to make this precise saying that the systole can shrink to $0$ as fast as it allows the sequence to be Plancherel (see Theorem \ref{geometric plancherel thm}). The conceptual idea is that quantum ergodicity is fundamentally a spectral phenomenon, so that Plancherel convergence seems to be the natural hypothesis under which it should hold. Moreover, we recover the case of multiplication operators studied in \cite{le2024quantum} in the degenerate case in which the propagation bound is $0$ and the kernels we consider are singular and supported on the diagonal in $X_n\times X_n$.

We also mention a recent result of Anantharaman and Saha \cite{anantharaman2026quantum}, who prove quantum ergodicity for pseudolocal operators on hyperbolic surfaces in the level aspect. Their theorem applies to sequences of closed hyperbolic surfaces which are Benjamini--Schramm convergent to the hyperbolic plane with a uniform lower bound on the systole and a uniform spectral gap, and allows operators with distributional kernels of uniformly finite propagation. In contrast, our theorem applies to finite-volume surfaces with cusps and replaces the Benjamini--Schramm and uniform discreteness assumptions by Plancherel convergence. As is customary in the finite-volume setting, our conclusion is formulated in terms of the quantum deviation rather than the variance. At the level of observables, our theorem is stated for measurable kernels satisfying uniform propagation conditions, and does not cover the full class of distributional kernels considered in \cite{anantharaman2026quantum}. Nevertheless, our strategy consists in disintegrating finite-propagation kernels into distributional components supported on $\{(x,y):d(x,y)=r\}$, and establishing the required quantum ergodicity estimate for these distributional kernels (cf. Proposition \ref{prop: first big reduction}).

We conclude with a comment on the restriction to compact intervals in $(1/4,\infty)$. A key property of Plancherel sequences is that, for any compact interval $I\subset[0,\infty)$, the spectral mass satisfies (cf. Proposition \ref{prop: spectral convergence})
$$\frac{N(X_n,I)+M(X_n,I)}{\operatorname{vol}(X_n)}
\longrightarrow \mu_{Pl}(I).$$
If $I\subset(1/4,\infty)$, then $\mu_{Pl}(I)>0$, and hence $N(X_n,I)+M(X_n,I)$ is asymptotically comparable to $\operatorname{vol}(X_n)$. We can therefore pass between estimates normalized by spectral mass and those normalized by volume. In contrast, if $I\subset[0,1/4]$, then $\mu_{Pl}(I)=0$, so the spectral mass is asimptotically $o(\operatorname{vol}(X_n))$. Consequently, volume-normalized estimates do not directly provide the spectral estimates required for $\operatorname{Dev}_{X_n,I}$, and in particular the quantitative bound of Proposition \ref{prop: geometric bound} is insufficient to control the quantum mean absolute deviation in this regime.\newline

\textbf{Structure of the paper: } In section \ref{preliminaries} we introduce the notation we use throughout the article and recall the necessary results about the harmonic analysis on hyperbolic surfaces. In section \ref{spectral convergence} we develop the notion of Plancherel convergence, provide a geometric interpretation by means of the Selberg Trace Formula, compare it with Benjamini--Schramm convergence and prove that Plancherel sequences satisfy the spectral convergence portrayed in (\ref{plancherel measure of interval}). Section \ref{quantum ergodicity and smooth kernels} is dedicated to the proof of the Quantum Ergodicity Theorem.\newline

\textbf{Acknowledgments.} I am deeply grateful to my advisor, Anton Deitmar, for introducing me to Plancherel convergence and the applications of the trace formula, for encouraging me to pursue this project, and for his constant support throughout its development. I also thank Carsten Peterson for teaching me much about quantum ergodicity at various stages of this work, and for his patience in listening to and discussing my ideas. Finally, I am particularly grateful to Tuomas Sahlsten for fruitful discussions concerning an issue in an earlier paper and for his suggestions on how to resolve it, as well as for his encouragement in the final stages of this project.

This work was carried out during my time as a PhD candidate at the University of Tübingen.

\section{Preliminaries}\label{preliminaries}

We fix some notation that will be used throughout the paper:
\begin{itemize}
    \item $G=\PSL_2(\R)$
    \item $K=\operatorname{PSO}(2)$
    \item $\Hyp=G/K$
    %\item $o=eK\in\Hyp$ is the projection of the identity element. In the upper-half plane model it corresponds to the point $i$.
    % \item If $\Gamma\le G$ is a torsion free lattice, we denote by $X=\Gamma\setminus\Hyp$ the associated hyperbolic surface. 
    % \item $\cu(X)=\{\ca_1,\dots,\ca_q\}$ is the set of cusps of $X$.
    % \item For $s\in\R$ and $p\in\{1,\dots, q\}$, 
    % $$E_p(s):= E_{\ca_p}\left(\cdot,\frac{1}{2}+is\right).$$
    % \item $\varphi_X:\C\to \C$ is the scattering determinant.
    % \item $\{\lambda_j\}_{j=0}^N$ are Laplace eigenvalues composing the discrete $L^2$-spectrum.
    % \item $\{\psi_j\}_{j=0}^N$ are Laplace eigenfunctions normalized to be of $L^2$-norm 1. 
    \item $\tau:[0,\infty)\to[1/4,\infty)$, $\tau(s):=\frac{1}{4}+s^2$.
    % \item $\phi_s:\Hyp\to\C$ is the spherical function with spectral parameter $s\in[0,\infty)$. 
    % \item $I\subset(1/4,\infty)$ is an interval.
\end{itemize}

\subsection{Cusps}\label{cusps}

Let $\Gamma\le G$ be a torsion free lattice and $X=\Gamma\setminus\Hyp$ be the associated hyperbolic surface. We say that $x\in\partial\Hyp$ is a \textit{parabolic fixed point} if there exists a parabolic element $\gamma\in\Gamma$ such that $\gamma(x)=x$. A \textbf{cusp} of $X$ is the $\Gamma$-orbit of a parabolic fixed point.
We denote by $\mathcal{C}(X)=\{\ca_1,\dots,\ca_q\}$ the set of cusps of $X$. For $p=1,\dots ,q$, let $x_p$ be a parabolic fixed point such that $\ca_p=\Gamma x_p$. The \textit{stabilizer} of $x_p$ is an infinite cyclic group generated by a parabolic element, i.e. there exists $\gamma_p\in\Gamma$ such that 
$$\Gamma_p=\{\gamma\in\Gamma\ :\ \gamma. x_p=x_p\}=\la\gamma_p\ra.$$

Henceforth, we assume to have fixed a parabolic fixed point for each cusp. For each $k=1,\dots, q$
there exists $\sigma_p\in G$ such that 
$$\sigma_p (\infty)=x_p,\qquad \sigma_p^{-1}\gamma_p\sigma_p=\begin{pmatrix}
    1 & 1 \\ 0 & 1
\end{pmatrix}.$$
Given $\upsilon>0$, consider the half-plane 
$$\cu(\upsilon)=\{x+iy\ : y\ge \upsilon\}$$
We denote
$$X_p(\upsilon):=\Gamma_p\setminus\sigma_p(\cu(\upsilon))$$
and we call $X_p(\upsilon)$ a \textit{cuspidal zone}. Note that $X_p(\upsilon)$ does not depend on the choice of the parabolic fixed point $x_p$. 
% In fact, if $\tilde{p}_j$ is in the $\Gamma$-orbit of $p_j$, then $\tilde p_j=\tilde\gamma p_j$ for some $\tilde\gamma\in\Gamma$ and we can take
% $$\tilde\sigma_j=\tilde\gamma\sigma_j,\qquad \tilde\gamma_j=\tilde\gamma\gamma_j\tilde\gamma^{-1},$$
% so that 
% $$\Gamma\setminus\sigma_j(\cu(t))=\Gamma\setminus\tilde\sigma_j(\cu(t))=X_j(t).$$
For $\upsilon\ge1$ the cuspidal zones are disjoint (cf. \cite[Thm. 4.4.6.(i)]{buser2010geometry}), so that we can decompose $X$ as 
\begin{equation}\label{surface decomposition}
    X=X(\upsilon)\cup\bigcup_{p=1}^q X_p(\upsilon),
\end{equation}
where $X(\upsilon)$ is compact with boundary consisting of disjoint horocycles 
$$c_p(\upsilon)=\Gamma\setminus\sigma_p (c(\upsilon)),\qquad \text{with }\quad c(\upsilon)=\{x+i\upsilon\ :\ x\in\R\}$$
for $p=1,\dots,q$. 

Now we discuss how functions $f$ on $X$ admit a Fourier series decomposition at each cusp. %Notice that each cusp cut at height $t$ is isometric to 
% $$\mathcal{C}_t=\Gamma_\infty\setminus\{ x+iy\in\Hyp\ :\ 0\le x\le1,\ y\ge t\},$$
% where $\Gamma_\infty=\la z\mapsto z+1\ra$. 
%Also, 
%$$\text{Vol}(\mathcal{C}(Y))=\frac{1}{Y}.$$

Let $p=1,\dots,q$. Given $f\in L^2(X)$ we denote $f^p(z):=f(\sigma_p(z))$. Then $f^p(z+1)=f^p(z)$. Since $f^p$ is 1-periodic, it admits a Fourier series decomposition
$$f^p(z)=f^p(x+iy)=\sum_{n=0}^\infty f^p_n(y)e^{inx},$$
where the equality is to be intended in the $L^2$ sense. If $f\in C^\infty(X)$, the Fourier series converges locally uniformly with all derivatives to $f^p$. This will be the case of Eisenstein series in Section \ref{Eisenstein series}.  

For future reference we provide here a formula for the number of cusps of a normal covering of $X$. 

\begin{lemma}\label{lem: n of cusps formula}
       Let $\{\ca_1,\dots,\ca_q\}$ be the set of cusps of $X$. for every $p\in\{1,\dots,q\}$ let $x_p$ be a parabolic fixed point such that $\ca_p=\Gamma x_p$ and let $\Gamma_p$ be the stabilizer in $\Gamma$ of $x_p$. Let $\Gamma'\trianglelefteq\Gamma$ be a normal subgroup such that $[\Gamma:\Gamma']<\infty$. Denote $X':=\Gamma'\setminus\Hyp$. Then the number of cusps of $X'$ is 
       $$q'=[\Gamma:\Gamma']\sum_{p=1}^q\frac{1}{[\Gamma_p:\Gamma_p\cap\Gamma']}.$$
\end{lemma}
    \begin{pc}
        fix $p\in\{1,\dots,m\}$. The set of cusps in $X'$ covering $\ca_p$ is in bijection with $\Gamma'\setminus \Gamma /\Gamma_p$. Since $\Gamma'$ is normal in $\Gamma$, we have $\Gamma'\setminus\Gamma /\Gamma_p\cong \Gamma /(\Gamma_p\Gamma')$, so that the number of cusps in $X'$ covering $\ca_p$ is
        $$\Big|\Gamma/(\Gamma_p\Gamma')\Big|=[\Gamma:\Gamma_p\Gamma'].$$
        By multiplicativity of the index $[\Gamma:\Gamma']=[\Gamma:\Gamma_p\Gamma']\cdot[\Gamma_p\Gamma':\Gamma']$, so that 
        $$[\Gamma:\Gamma_p\Gamma']=\frac{[\Gamma:\Gamma']}{[\Gamma_p\Gamma':\Gamma']}.$$
        Now, 
        $$[\Gamma_p\Gamma':\Gamma']=\Big|\Gamma_p\Gamma_n/\Gamma_n\Big|=\Big|\Gamma_p /(\Gamma_p\cap\Gamma')\Big|=[\Gamma_p:\Gamma_p\cap\Gamma'],$$
        where the second equality is a consequence of the second isomorphism Theorem, according to which $\Gamma_p\Gamma'/\Gamma'\cong\Gamma_p /(\Gamma_p\cap\Gamma')$.
        Summarizing, for every fixed cusp $\ca_p$ in $X$, the number of cusps in $X'$ covering it is $\frac{[\Gamma:\Gamma']}{[\Gamma_p:\Gamma_p\cap\Gamma']}$. Summing over all cusps of $X$ concludes the proof of the Lemma.
    \end{pc}

\subsection{Eisenstein Series and Scattering Matrix}\label{Eisenstein series}

To each cusp $\ca_p\in\mathcal{C}(X)$ we associate the Eisenstein series 
$$E_p(z,s)=\sum_{\gamma\in\Gamma_p\setminus\Gamma}\left(\operatorname{Im} \sigma_p^{-1}\gamma. z\right)^{s},$$
where the series converges for $z\in X$ and $s\in\C$ with $\operatorname{Re} s>1$. The Eisenstein series does not depend on the choice of the parabolic fixed point $x_p$ for the cusp $\ca_p$. 
% In fact, if $\tilde p_j=\tilde\gamma p_j$ for some $\tilde\gamma\in\Gamma$, then 
% $$\sum_{\gamma\in\la\gamma_j\ra\setminus\Gamma}(\operatorname{Im}\sigma_j^{-1}\gamma z)^s=\sum_{\gamma\in\la\gamma_j\ra\setminus\Gamma}(\operatorname{Im}\tilde\sigma_j^{-1}\tilde\gamma\gamma z)^s$$
% $$=\sum_{\gamma\in\la\tilde\gamma_j\ra\setminus\Gamma}(\operatorname{Im}\tilde\sigma_j^{-1}\gamma z)^s,$$
% where the last step follows from observing that $\gamma\in\la\gamma_j\ra\setminus\Gamma\iff\tilde\gamma\in\la\tilde\gamma_j\ra\setminus\Gamma$.
For each $z\in X$, the Eisenstein series admits a meromorphic continuation $s\mapsto E_p(z,s)$ to the whole complex plane (see \cite[Sections 3.3+3.4]{iwaniec2021spectral}). The Eisenstein series $z\mapsto E_p(z,s)$ is an eigenfunction of the Laplacian on $X$ with eigenvalue $\lambda=s(1-s)$.

For $s\in\C$ and $j,k=1,\dots, q$, we can expand the Eisenstein series $E_{j}(\cdot,s)$ associated with the cusp $\ca_j$ in its Fourier series with respect to the cusp $\ca_k$. According to \cite[Thm. 3.4 + (6.18)]{iwaniec2021spectral}, there exists a function $\Phi_{j,k}:\C\to\C$ such that 
$$E_j^{k}(z,s)=E_j(\sigma_k(z),s)=\delta_{jk} y^s+\Phi_{j,k}(s) y^{1-s}+\sum_{n\ne 0}^\infty f^k_n(s,y)e^{inx}$$
for some functions $f^k_n(s,y)$ representing the coefficients of the non-zero Fourier modes. 

The $q\times q$ matrix $\Phi(s)$ with $j,k$-entries $\Phi_{j,k}(s)$ is called the \textbf{scattering matrix} of $X$. The determinant of $\Phi(s)$ is called the \textbf{scattering determinant} and it is denoted by $\varphi(s)$. We recall here some properties of the scattering matrix (see \cite[Theorem 6.6]{iwaniec2021spectral}:
    The scattering matrix satisfies the functional equation 
    \begin{equation}\label{scattering functional equation}
        \Phi(s)\Phi(1-s)=I.
    \end{equation}
    When $\operatorname{Re}s=\frac{1}{2}$, the scattering matrix $\Phi(s)$ is unitary, i.e.
    \begin{equation}\label{scattering is unitary}
        \Phi(s)\overline{\Phi^t(s)}=I.
    \end{equation}

\subsection{Polar coordinates}

Fix $x\in\Hyp$ and a vector $v\in T_x\Hyp$. Polar coordinates based at $(x,v)$ can be defined as follows. For any $z\in\Hyp\smallsetminus\{x\}$ there exists a unique geodesic $\gamma$ with constant speed one such that $\gamma(0)=x$ and passing through $z$. Let $r(z):=d(x,z)$ and $\theta(z)\in\mathbb{S}^1$ be the oriented angle from $v$ to $\dot\gamma(0)$. Then $(r,\theta)=(r(z),\theta(z))$ are the polar coordinates of $z$ relative to the choice of the base point $x$ and tangent vector $v$. In polar coordinates, the hyperbolic metric admits the expression 
$$dz^2=dr^2+\sinh rd\theta^2.$$

If $f\in L^1(\Hyp)$,then 

\begin{equation}\label{eq: polar coordinates}
    \int_\Hyp f(z)dz=\int_0^\infty\int_{\mathbb{S}^1} f(r,\theta)\sinh rd\theta dr.
\end{equation}

\subsection{Spherical functions, Invariant integral operators and Eigenvalue Functions}\label{selberg and abel transform}

The spectrum of the Laplacian on $\Hyp$ is $\left[\frac{1}{4},\infty\right)$. A $K$-bi-invariant function on $G$ (equivalently, a left $K$-invariant function on $\Hyp$) which is also an eigenfunction of the Laplacian on $\Hyp$ is called a \textbf{spherical function}. For any $s\in\left[0,\infty\right)$, there exists a unique spherical function $\phi_s$ such that $\phi_s(e)=0$ and 
$$\Delta\phi_s=\lambda_s\phi_s,$$
where $\lambda_s=\frac{1}{4}+s^2$. The above statement is a special case of a general Theorem of Harish-Chandra \cite{harish1958spherical}. A proof in our setting can be found in \cite[Lemma 1.12]{iwaniec2021spectral}.

Now, any $k\in C^\infty_c(\R)_{\operatorname{ev}}$ ( Equivalently, $k\in C^\infty_c(K\setminus G/K)$) defines an \textit{invariant integral operator} 
$$T_k(f)=\int_\Hyp k(\cdot,w)f(w)dw,$$ 
where by abuse of notation we write $k(z,w)=k(d(z,w))$. We call $k$ the \textit{kernel} of $T_k$.

Every eigenfunction of the Laplacian on $\Hyp$ is also an eigenfunction of all invariant integral operators with compactly supported kernels. More precisely (\cite[1.14]{iwaniec2021spectral}), if $k\in C^\infty_c(\R)_{\operatorname{ev}}$ and $\Delta f=\lambda_s f$, then 
\begin{equation}\label{invariant integral eigenvalue}
    T_k f=\widehat{k}(s)f,
\end{equation}
where 
$$\widehat{k}(s):=\int_0^\infty k(r)\phi_s(r)\sinh(r)dr$$
is the \textbf{spherical transform} of the kernel $k$ defined by integrating in polar coordinates against the spherical function with spectral parameter $s$. 
% Equivalently, 
% $$\widehat{k}(s)=\int_\Hyp k(z,i)\phi_s(z)dz.$$
\begin{rem}\label{spherical transform for non smooth kernels}
    Notice that (\ref{invariant integral eigenvalue}) can be extended to any invariant integral operator with kernel $k\in L^\infty(\R_+)$ which is compactly supported by means of dominated convergence\footnote{For any $k\in L^\infty(\R_+)$ and $\varepsilon>0$ there exists a sequence of kernels $k_n\in C^\infty(\R)_{\operatorname{ev}}$ whose restriction to $\R_+$ converges pointwise (almost surely) to $k$ and such that $k_n(r)\le k(r)+\varepsilon$ for every $r\ge 0$, so that by dominated convergence
    $\widehat{k_n}(s)\to\widehat{k}(s)$ and $T_{k_n}f(x)\to T_k f(x)$.}. 
\end{rem}

The spherical transform (also called in this setting the \textit{Selberg transform}) can equivalently be computed as the Fourier transform

$$\widehat{k}(s)=\int_\R e^{ist}g(t)dt$$
of the Abel transform 
$$g(t)=\mathcal{A}(k)(t)=\sqrt{2}\int_{|t|}^\infty\frac{k(x)\sinh x}{\sqrt{\cosh x-\cosh t}}dx.$$ 
The inverse spherical/Selberg transform is the inverse Abel transform
$$\mathcal{A}^{-1}(g)(\rho)=-\frac{1}{\sqrt{2}\pi}\int_{|\rho|}^\infty\frac{g'(t)}{\sqrt{\cosh t-\cosh\rho}}dt$$
of the inverse Fourier transform 
$$g(t)=\frac{1}{2\pi}\int_\R e^{-itr}h(r)dr.$$

% \begin{thm}[\cite{iwaniec2021spectral}\label{eigenvalue theorem}, Thm. 1.14 + 1.16] Every eigenfunction of the Laplacian on $\Hyp$ is also an eigenfunction of all invariant integral operators. More precisely, given $k\in C^\infty_c(\R)_{\operatorname{ev}}$ and  $f\in C^\infty(\Hyp)$ a $\Delta$-eigenfunction with eigenvalue $\lambda\in\C$, then
% $$T_k f=h(r_\lambda)f,$$
% where the eigenvalue $h(r_\lambda)$ is given by the Selberg transform 
% $$h(r_\lambda)=\mathcal{S}(k)(r_\lambda)$$
% and $r_\lambda\in\R$ is defined by $\lambda=\frac{1}{4}+r_\lambda^2$.
% We call $h$ the \textbf{eigenvalue function} of $T_k$. 
% \end{thm}

The condition that the kernel $k$ is compactly supported is not essential in order for (\ref{invariant integral eigenvalue}) to hold. However, a control on the decay of $k$ is necessary. It is more convenient to express sufficient conditions in terms of its spherical transform.

Let $h:\C\to\C$ satisfy the following conditions:
\begin{equation}\label{admissible eigenvalue functions}
\begin{split}
    &\text{i) } h \text{ is even}, \\
    &\text{ii) } h \text{ is holomorphic in the strip } |\operatorname{Im} z|\le\frac{1}{2}+\eps,\\
    &\text{iii) } h(z)=O\left((1+|z|)^{-2-\eps}\right).
\end{split}
\end{equation}

The inverse Selberg transform gives a function $k\in C^\infty(\R)_{\operatorname{ev}}$ which is the kernel of an invariant integral operator $T_k$. We call a function $h$ satisfying conditions i)-iii) an \textbf{admissible eigenvalue function}. The spherical transform $\widehat{k}$ of $k$ coincides with the eigenvalue function $h$.

\subsection{Selberg Trace Formula}

For a thorough discussion of Selberg Trace Formula for non-compact surfaces with finite area we refer to \cite[Chapter 10]{iwaniec2021spectral}.  Let $X=\Gamma\setminus\Hyp$ be a finite-area hyperbolic surface, and let $\mathcal{F}\subset\Hyp$ be a fundamental domain of $X$. Let $h$ be an admissible eigenvalue function, $k$ the associated kernel and $T_k$ the invariant integral operator with kernel $K(x,y):=\sum_{\gamma\in\Gamma}k(x,\gamma .y)$. 
% Given $f\in L^2(X)$, following standard trace formula arguments (see e.g. \cite[Section 10.1]{iwaniec2021spectral} or \cite[Theorem 9.3.2]{deitmar2014principles}), we obtain

% \begin{equation}\label{diverging trace}
%     T_kf(z)=\int_{\mathcal{F}}\sum_{\gamma\in\Gamma}k(z,\gamma w) f(w) dw.
% \end{equation}

If $X$ has cusps, the integral 
$$\int_\F\sum_{\gamma\in\Gamma}k(z,\gamma z)dz$$
diverges, so that $T_k$ is not trace class. The decomposition of $X$ as disjoint union of a compact surface and cuspidal zones (\ref{surface decomposition}) induces a decomposition 
$$\F=\F(\upsilon)\cup\bigcup_{p=1}^q \F_p(\upsilon).$$
Selberg's idea was to compute the trace on $\F(\upsilon)\subset\F$ both spectrally and geometrically and then study the asymptotic as $\upsilon\to\infty$. Diverging terms in $\upsilon$ then cancel each other. The outcome of this procedure is known as Selberg's trace formula. 

We denote $\Gamma_{\hyp}=\{\gamma\in\Gamma\ : \ \gamma\ \text{hyperbolic}\}$.  Let
$$0=\lambda_0<\lambda_1\le\dots$$
be the sequence of eigenvalues in the discrete Laplace spectrum of $X$. For all $j\in\N$, let $s_j$ be defined by $\lambda_j=\frac{1}{4}+s_j^2$. Since $X$ is smooth, there are no elliptics element in $\Gamma$. Selberg Trace Formula for $X$ reads as follows (\cite[Sections 10.2-10.5]{iwaniec2021spectral}):

 \begin{equation*}
     \begin{split}
         \sum_j h(s_j)+\frac{1}{4\pi}\int_\R h(s)\frac{-\varphi'_X}{\varphi_X}\left(\frac{1}{2}+is\right)ds=&\frac{\operatorname{vol}(X)}{4\pi}\int_\R h(s)\tanh(\pi s)ds\\
         +&\int_\F\sum_{\gamma\in\Gamma_{\hyp}} k(z,\gamma. z)dz\\
         +&\left(g(0)\gamma_E-\int_0^\infty\log\left(\sinh \frac{t}{2}\right)g'(t)dt\right)q\\
         -&\frac{h(0)}{4}\operatorname{Tr}\left(\Phi\left(\frac{1}{2}\right)\right),
     \end{split}
 \end{equation*}
 where $q$ is the number of cusps of $X$, $\gamma_E$ is the Euler-Mascheroni constant and $g$ is defined as in Section \ref{selberg and abel transform}.
 In the RHS, the first and second lines represent the contribution to the trace of the identity and hyperbolic motions respectively. The third line represents the contribution to the trace of the parabolic motions, after removing the diverging terms. The fourth line comes from the computation of the spectral trace.

\newpage

\section{Spectral convergence}\label{spectral convergence}

We provide some background for Plancherel convergence of cocompact lattices and explain the main difficulties in extending the definition to the nonuniform case. Let $G$ be a connected semisimple Lie group of non-compact type and fix a Haar measure $\mu$ on $G$. Let $\Gamma\le G$ be a uniform lattice and $\widehat G$ the unitary dual of $G$. For any unitary representation $(\eta, H_\eta)$ of $G$ and $f\in C^\infty_c(G)$ we denote by $\eta(f)$ the bounded operator on $H_\eta$ defined as a Bochner integral by $\eta(f):=\int_G f(x)\eta(x) dx\in\mathcal{B}(H_\eta)$. The right regular representation $R$ on $L^2(\Gamma\setminus G)$ defined by $R_y\phi(x):=\phi(xy)$ decomposes into a direct sum of unitary irreducible representations
\begin{equation*}
    L^2(\Gamma\setminus G)\cong \bigoplus_{\pi\in\widehat G} N_\Gamma(\pi)\pi,
\end{equation*}
with finite multiplicities $N_\Gamma(\pi)$. Let $\delta_\pi$ denote the Dirac measure for $\pi\in\widehat G$. The measure on $\widehat G$ defined by 
    \begin{equation*}
        \mu_\Gamma=\sum_{\pi\in\widehat G}N_{\Gamma}(\pi)\delta_\pi
    \end{equation*}
    is called the \textit{spectral measure} associated with $\Gamma$. 
    Let $\mu_{\Pl}$ be the Plancherel measure on $\widehat{G}$. A sequence $(\Gamma_n)_{n\in\N}$ of uniform lattices in $G$ is called \textit{Plancherel convergent} (or a \textit{Plancherel sequence}) if for every $f\in C^\infty_c(G)$
    \begin{equation*}
        \frac{1}{\text{vol}(\Gamma_n\setminus G)}\mu_{\Gamma_n}(\hat f)\xrightarrow[]{n\to\infty}\mu_{\Pl}(\hat f),
    \end{equation*}
    where $\hat f(\pi):=\operatorname{tr}\pi(f)$. In particular, notice that $\mu_{\Gamma_n}(\hat f)=\operatorname{tr} R_n(f)$. A difficulty in defining Plancherel convergence for sequences of nonuniform lattices is that the right regular representation is not trace class. That is, if $\Gamma\le G$ is a nonuniform lattice and $f\in C^\infty_c(G)$, then the operator $R(f):L^2(\Gamma\setminus G)\to L^2(\Gamma\setminus G)$ need not be trace class, so requiring 
    $$\frac{\operatorname{tr} R_n(f)}{\vol(\Gamma_n\setminus G)}\to\mu_{\Pl}(\hat f)$$
    is meaningless. The Selberg trace formula provides another perspective on Plancherel convergence. In the cocompact case, $\mu_\Gamma(\hat f)$ is precisely the spectral side of the trace formula, whereas $\mu_{\Pl}(\hat f)$ is the contribution of the identity element to orbital integrals in the geometric side. Thus, Plancherel convergence can be viewed as the statement that, after normalization by the volume, the spectral side converges to the identity contribution. This viewpoint continue to make sense for nonuniform lattices, even though the regular representation is no longer trace class. This is the interpretation we employ in the following for hyperbolic surfaces.

\subsection{Definition and geometric interpretation}

Let $(X_n)_{n\in\N}$ be a sequence of finite-area hyperbolic surfaces. For each $n\in\N$, let
$$0=\lambda_0^{(n)}<\lambda_1^{(n)}\le \lambda_2^{(n)}\le\dots$$
be the sequence of eigenvalues in the discrete Laplace spectrum of $X_n$. For all $j,n\in\N$, let $r_j^{(n)}$ be defined by $\lambda_j^{(n)}=\frac{1}{4}+\left(s_j^{(n)}\right)^2$.
We say that $(X_n)_{n\in\N}$ is a \textbf{Plancherel sequence} if for any $k\in C^\infty_c(\R)_{\operatorname{ev}}$
$$\frac{1}{\text{vol}(X_n)}\left(\sum_j h\left(s^{(n)}_j\right)+\frac{1}{4\pi}\int_\R h(s)\frac{-\varphi_{X_n}'}{\varphi_{X_n}}\left(\frac{1}{2}+is\right)ds\right)\xrightarrow[]{n\to\infty}\mu_{\operatorname{Pl}}(h),$$
where $h=\widehat{k}$ and 
$$\mu_{\operatorname{Pl}}(h)=\frac{1}{4\pi}\int_\R h(s)\tanh(\pi s)sds.$$

Classical examples of Plancherel sequences of cocompact lattices are towers of normal subgroups (see \cite{degeorge1978limit}). We will show that this is the case for nonuniform lattices as well in Example \ref{towers are plancherel}. Our first goal is to provide a geometric interpretation of Plancherel sequences, reminiscent of the case of compact surfaces (see \cite[Prop. 2.9]{deitmar2019benjamini}), with the help of Selberg trace formula. In the following, we denote the number of cusps of $X_n$ by $q_n$. A first consequence of the definition is that the contribution of cusps must become negligible. this will later allow us to isolate the contribution of hyperbolic elements in the trace formula. 

\begin{lemma}\label{not too many cusps}
    Let $(X_n)_{n\in\N}$ be a Plancherel sequence of finite-area hyperbolic surfaces. Then 
    $$\frac{q_n}{\operatorname{vol}(X_n)}\xrightarrow[]{n\to\infty}0.$$
\end{lemma}
\begin{proof}
Let us denote  
\begin{equation*}
    \begin{split}
        \textbf{(A)}_n=&\int_{X_n} \sum_{\Gamma_{n,\hyp}} k(z,\gamma .z)dz,\\
        \textbf{(B)}_n=&\left(g(0)\gamma_E-\int_0^\infty\log(\sinh\frac{r}{2})g'(r)dr\right)q_n-\frac{h(0)}{4}\operatorname{Tr}\left(\Phi_n\left(\frac{1}{2}\right)\right).
    \end{split}
\end{equation*}

By means of the Selberg Trace Formula we see that a sequence is Plancherel if and only if for any $k\in C^\infty_c(\R)_{\operatorname{ev}}$
    \begin{equation}\label{playing with trace formula}
    \begin{split}
        \frac{\textbf{(A)}_n+\textbf{(B)}_n}{\operatorname{vol}(X_n)}\xrightarrow[]{n\to\infty}0.
    \end{split}
    \end{equation}
    \begin{claim}\label{claim super sayan}
    
        There exists a kernel $k_0\in C^\infty_c(\R)_{\operatorname{ev}}$ such that for each $n\in\N$
    \begin{equation}\label{conditions in claim}
        \begin{split}
            &\textbf{(A)}_n>0,\\
            &g_0(0)\gamma_E-\frac{h_0(0)}{4}-\int_0^\infty\log(\sinh\frac{r}{2})g_0'(r)dr>0.
        \end{split}
    \end{equation}
    \end{claim}
Now we show the Lemma assuming the Claim, and only prove it afterwards. Recall that $\Phi_n\left(\frac{1}{2}\right)$ is unitary (\ref{scattering is unitary}), so that $\left|\operatorname{Tr}\left(\Phi_n\left(\frac{1}{2}\right)\right)\right|\le q_n$. Then, for each $n\in\N$
$$0\le\left(g_0(0)\gamma_E-\frac{h_0(0)}{4}-\int_0^\infty\log(\sinh\frac{r}{2})g_0'(r)dr>0\right)q_n\le\operatorname{Re}\textbf{(B)}_n.$$
Hence, we get by (\ref{playing with trace formula}) that both 
$$\frac{\textbf{(A)}_n}{\operatorname{vol}(X_n)}\xrightarrow[]{n\to\infty}0\qquad\text{and}\qquad\frac{\operatorname{Re}\textbf{(B)}_n}{\operatorname{vol}(X_n)}\xrightarrow[]{n\to\infty}0.$$
This implies that 
    \begin{equation*}
    \begin{split}
        &\frac{q_n}{\operatorname{vol(X_n)}}\left(g_0(0)\gamma_E-\frac{h_0(0)}{4}-\int_0^\infty\log(\sinh\frac{r}{2})g_0'(r)dr\right)\xrightarrow[]{n\to\infty}0\\
        \iff&\frac{q_n}{\operatorname{vol(X_n)}}\xrightarrow[]{n\to\infty}0.
    \end{split}
    \end{equation*}

     We are now left with the proof of the Claim.
     
    \textit{Proof of Claim \ref{claim super sayan}.}
    In order to provide a test kernel whose hyperbolic contribution is always nonnegative while the parabolic coefficient is strictly positive, we start by prescribing the eigenvalue function. Let $h(s)=e^{-\frac{s^2}{4\pi}}$, so that $g(r)=e^{-r^2}$. Notice that $\gamma_E>\frac{1}{4}$ and that $g'(r)<0$, so that
    $$C(g):=g(0)\gamma_E-\frac{h(0)}{4}-\int_0^\infty\log(\sinh\frac{r}{2})g'(r)dr>0.$$
    Let 
    $$k(\rho)=\mathcal{A}^{-1}(g)(\rho)=-\frac{1}{\sqrt{2}\pi}\int_\rho^\infty\frac{g'(r)}{\sqrt{\cosh r-\cosh \rho}}dr.$$
    Since $g'(r)<0,$ we get $k(\rho)>0$, so that for each $n\in\N$
    $$\int_{\F_n}\sum_{\Gamma_{n,\hyp}}k(z,\gamma. z)dz>0.$$
    The kernel $k$ satisfies the required properties. However, it is not compactly supported. To fix this, we approximate $k$ with a sequence of suitable compactly supported kernels. For $j\in\N$ let $\psi_j\in C_c^\infty(\R)$ be a cut-off function such that 
    \begin{equation*}
        \begin{split}
            &\text{a) }\psi_j(x)=1\ \text{for } |x|\le j,\\
            &\text{b) }\psi_j(x)=0\ \text{for } |x|\ge j+1.\\
            &\text{c) }\psi_j(x)\ge0\ \text{for all }x\in\R.\\
            &\text{d)} |\psi_j'(x)|\le M\ \text{for some fixed $M>0$.}
        \end{split}
    \end{equation*}
    Consider the sequence $(g_j)_{j\in\N}$, where $g_j=g\psi_j$ for each $j\in\N$. Then,
    $$C(g_j)\xrightarrow[]{j\to\infty}C(g)>0.$$
    Choose $j_0\in\N$ large enough so that $C(g_{j_0})>0$. Notice that 
    $$k_{j_0}(\rho)=\mathcal{A}^{-1}(g_{j_0})(\rho)\ge 0\qquad \text{for all }\rho\in[0,\infty),$$
    so that 
    $$\int_{\F_n}\sum_{\Gamma_{n,\hyp}}k_{j_0}(z,\gamma. z)dz>0. $$
    Also, $k_{j_0}$ is compactly supported. This concludes the proof of the Claim. 
\end{proof}

\begin{propx}\label{first geometric plancherel interpretation}
    A sequence $(X_n)$ of finite-area hyperbolic surfaces is Plancherel if and only if
    \begin{itemize}
        \item $\frac{q_n}{\operatorname{vol}(X_n)}\xrightarrow[]{n\to\infty}0$, and
        \item for any $k\in C^\infty_c(\R)_{\operatorname{ev}}$,
    $$\frac{1}{\operatorname{vol}(X_n)}\int_{\mathcal{F}_n}\sum_{\Gamma_{n,\hyp}}k(z,\gamma .z)dz\xrightarrow[]{n\to\infty}0.$$
    \end{itemize} 
\end{propx}
\begin{proof}
    Recall that $\Phi_n\left(\frac{1}{2}\right)$ is unitary, so that $\left|\operatorname{Tr}\left(\Phi_n\left(\frac{1}{2}\right)\right)\right|\le q_n$. Hence, 
    if $\frac{q_n}{\operatorname{vol}(X_n)}\xrightarrow[]{n\to\infty}0$, then for any $k\in C_c^\infty(\R_+)$
    $$\frac{q_n}{\operatorname{vol}(X_n)}\left(g(0)\gamma_E-\int_0^\infty\log(\sinh\frac{r}{2})g'(r)dr\right)\xrightarrow[]{n\to\infty}0$$
    and 
    $$\frac{h(0)}{4\operatorname{vol}(X_n)}\operatorname{Tr}\left(\Phi_n\left(\frac{1}{2}\right)\right)\xrightarrow[]{n\to\infty}0.$$
    If also 
    $$\frac{1}{\operatorname{vol}(X_n)}\int_{\mathcal{F}_n}\sum_{\Gamma_{n,\hyp}}k(z,\gamma. z)dz\xrightarrow[]{n\to\infty}0,$$
    then (\ref{playing with trace formula}) holds and the sequence is Plancherel. 

    For the converse direction, assume that the sequence is Plancherel. From Lemma \ref{not too many cusps} we already know that 
    $$\frac{q_n}{\operatorname{vol}(X_n)}\xrightarrow[]{n\to\infty}0.$$
    
    Then (\ref{playing with trace formula}) reduces to
    $$\frac{1}{\operatorname{vol}(X_n)}\int_{\mathcal{F}_n}\sum_{\Gamma_{n,\hyp}}k(z,\gamma. z)dz\xrightarrow[]{n\to\infty}0$$
    for any $k\in C^\infty_c(\R)_{\operatorname{ev}}$.

\end{proof}

\begin{rem}\label{plancherel convergence and closed geodesics}
    Recall that any hyperbolic element in $\Gamma_n$ is a positive power of a uniquely determined primitive element $\gamma_0$, called the \textbf{primitive element underlying $\gamma$}. The contribution of hyperbolic elements to the trace is \cite[Thm. 11.4.3]{deitmar2014principles}
\begin{equation*}
   \int_\F \sum_{\Gamma_{n,\hyp}} k(z,\gamma .z)dz=\sum_{\substack{[\gamma]\\\gamma\in\Gamma_{n,\hyp}}}\frac{l_{\gamma_0}}{2\sinh(l_\gamma/2)}g(l_\gamma),
\end{equation*}
where the sum on the RHS runs over all conjugacy classes of hyperbolic elements. 
Moreover, the Abel transform of a kernel $k$ is an even function on $\R$ and it preserves the compactness of $k$. Since the Abel transform is invertible, we get that 
$$\mathcal{A}:k\in C^\infty_c(\R)_{\operatorname{ev}}\longrightarrow C_c^\infty(\R)_{\text{ev}}$$
is a bijection.
Therefore, as an immediate corollary of Proposition \ref{first geometric plancherel interpretation} we get that a sequence $(X_n)_{n\in\N}$ is Plancherel if and only if 
\begin{equation}\label{Plancherel condition on Abel transform}
    \begin{split}
        \bullet&\  \frac{q_n}{\operatorname{vol}(X_n)}\xrightarrow[]{n\to\infty}0,\text{ and}         \\
        \bullet&\text{ For any } g\in C^\infty_c(\R)_{\operatorname{ev}},\\
        &\qquad \frac{1}{\operatorname{vol}(X_n)}\sum_{\substack{[\gamma]\\\gamma\in\Gamma_{n,\hyp}}}\frac{l_{\gamma_0}}{2\sinh(l_\gamma/2)}g(l_\gamma)\xrightarrow[]{n\to\infty}0,
    \end{split}
\end{equation}
\end{rem}

\begin{lemma}\label{last plancherel lemma}
    Let $(X_n)_{n\in\N}$ be a sequence of finite-area hyperbolic surfaces. Then 
    $$\frac{1}{\operatorname{vol}(X_n)}\int_{\mathcal{F}_n}\sum_{\Gamma_{n,\hyp}}k(z,\gamma. z)dz\xrightarrow[]{n\to\infty}0$$
    for any $k\in C^\infty_c(\R)_{\operatorname{ev}}$ if and only if for all $R>0$
    $$\frac{1}{\operatorname{vol}(X_n)}\int_{\mathcal{F}_n}\sharp\left(\Gamma_{n,\hyp}. z\cap \overline{B_R}(z)\right)dz\xrightarrow[]{n\to\infty}0.$$
\end{lemma}
\begin{proof}
    $(\Rightarrow)$ Let $R>0$. Let $k\in C^\infty_c(\R)_{\operatorname{ev}}$ be such that $k(r)=1$ for all $r\le R$, and $k(r)\ge 0$. Therefore, 
    $$\frac{1}{\text{vol}(X_n)}\int_{\mathcal{F}_n}\sharp\left(\Gamma_{n,\hyp}.z\cap \overline{B_R}(z)\right)dz\le\frac{1}{\text{vol}(X_n)}\int_{\mathcal{F}_n}\sum_{\Gamma_{n,\hyp}}k(z,\gamma .z)dz\xrightarrow{n\to\infty}0.$$
    $(\Leftarrow)$ Let $k\in C_c^\infty(\R)_{\operatorname{ev}}$. Choose $R>0$ such that $\operatorname{supp}k\subseteq[-R,R]$. Then
    \begin{equation}
    \begin{split}
        \frac{1}{\text{vol}(X_n)}\int_{\mathcal{F}_n}\sum_{\Gamma_{n,\hyp}}k(z,\gamma .z)dz=&\frac{1}{\text{vol}(X_n)}\int_{\mathcal{F}_n}\sum_{\substack{ \Gamma_{n,\hyp}\\ d(z,\gamma. z)\le R}}k(z,\gamma .z)dz\\
        \le&\frac{||k||_\infty}{\text{vol}(X_n)}\int_{\mathcal{F}_n}\sharp\left(\Gamma_{n,\hyp}.z\cap \overline{B_R}(z)\right)dz\xrightarrow[]{n\to\infty}0.
    \end{split}
    \end{equation}
    
\end{proof}

Finally, we provide a characterization of Plancherel sequences which is the counterpart of \cite[Proposition 2.9]{deitmar2019benjamini} in the noncompact case.

\begin{thm}\label{geometric plancherel thm}
    A sequence $(X_n)$ of finite-area hyperbolic surfaces is Plancherel if and only if
    \begin{enumerate}[1)]
        \item $\frac{q_n}{\operatorname{vol}(X_n)}\xrightarrow{n\to\infty}0$, and
        \item for all $R>0$
         \begin{equation*}
        \frac{1}{\operatorname{vol}(X_n)}\int_{\mathcal{F}_n}\sharp\left(\Gamma_{n,\hyp}.z\cap \overline{B_R}(z)\right)dz\xrightarrow[]{n\to\infty}0,
    \end{equation*}
    \end{enumerate}
\end{thm}

\begin{proof}
    By Proposition \ref{first geometric plancherel interpretation} we get that $(X_n)_{n\in\N}$ is Plancherel if and only if $\frac{q_n}{\operatorname{vol}(X_n)}\xrightarrow{n\to\infty}0$ and for any $k\in C^\infty_c(\R)_{\operatorname{ev}}$
    $$\frac{1}{\text{vol}(X_n)}\int_{\mathcal{F}_n}\sum_{\Gamma_{n,\hyp}}k(z,\gamma. z)dz\xrightarrow[]{n\to\infty}0.$$
    By Lemma \ref{last plancherel lemma} this is equivalent to 
    $$\frac{1}{\text{vol}(X_n)}\int_{\mathcal{F}_n}\sharp\left(\Gamma_{n,\hyp}.z\cap \overline{B_R}(z)\right)dz\xrightarrow[]{n\to\infty}0,$$
    for all $R>0$, which proves the Theorem.
\end{proof}

\begin{rem}
    Let us denote $\Gamma^\star=\Gamma\smallsetminus \{\operatorname{id}\}$. Recall that a sequence of closed hyperbolic surfaces is Plancherel if and only if for all $R>0$ 
    $$\frac{1}{\vol X_n}\int_{\mathcal{F}_n}\sharp \Big(\Gamma_n^\star .z\cap\overline{B_R}(z)\Big)dy\xrightarrow[]{n\to\infty}0.$$
    This is false in the presence of cusps, for the reason that parabolic elements contribute too many orbit points when far enough into a cusp. However, we can retrieve a useful statement by applying a cutoff. That is, for any $\upsilon\ge 1$ 
    $$\frac{1}{\vol X_n}\int_{\mathcal{F}_n(\upsilon)}\sharp \Big(\Gamma_n^\star.z\cap \overline{B_R}(z)\Big) dz\xrightarrow[]{n\to\infty}0.$$
    This is a consequence of the fact that, while the stabilizer of a cusp contributes a lot or orbit points into ``its" cusp, it fails to contribute away from it. We make this precise. 
\end{rem}

\begin{lemma}\label{cusps and thin part}
    Let $X=\Gamma\setminus\Hyp$ be a hyperbolic surface of finite area with cusps $\{\ca_1\dots,\ca_q\}$. For each $k=1,\dots,q$, let $\gamma_k\in\Gamma$ be as defined in section \ref{cusps}. Then, for each $R>0$ there exists a constant $C_R$, only depending on $R$, such that for every $k=1,\dots,q$
    $$\operatorname{vol}\left(\{ z\in\F\ :\ d(z,\gamma_k .z)<R\}\right)\le C_R,$$
    where $\F$ is a fundamental domain for $X$.
\end{lemma}
\begin{proof}
    Fix $k\in\{1,\dots,q\}$ and consider $z\in \F$. Let $\sigma_k$ be defined as in section \ref{cusps}.
    Let $y(R)$ be small enough such that $d(w,w+1)>R$ whenever $\operatorname{Im}w< y(R)$. The set $Y_R=\{w\in\Hyp\ :\ 0\le\operatorname{Re}w\le1,\ \operatorname{Im}w\ge y(R)\}$ has volume $\frac{1}{y(R)}$. Therefore,
    $$\operatorname{vol}\left(\{ z\in\F\ :\ d(z,\gamma_k. z)<R\}\right)\le\operatorname{vol}(Y_R)=\frac{1}{y(R)}.$$
    In fact, 
    \begin{equation*}
        \begin{split}
            d(z,\gamma_k.z)=&d(\sigma_k^{-1}.z,\sigma_k^{-1}\gamma_k. z)\\
            =& d(\sigma_k^{-1} .z,\left(\begin{smallmatrix}1 & 0\\0 & 1\end{smallmatrix}\right)\sigma_k^{-1}.z)\\
            =&d(\sigma_k^{-1}.z,\sigma_k^{-1}.z+1),
        \end{split}
    \end{equation*}
    so that $d(z,\gamma_k .z)<R\iff d(\sigma_k^{-1}.z,\sigma_k^{-1}.z+1)<R\Rightarrow\sigma_k^{-1}.z\in Y(R)$.
\end{proof}

\begin{lemma}\label{lem: Plancherel and parabolic contribution}
    Let $(X_n)_{n\in\N}$ be a Plancherel sequence of finite-area hyperbolic surfaces. Then for every $R>0$ and $\upsilon\ge 1$ 
    $$\frac{1}{\vol X_n}\int_{\mathcal{F}_n(\upsilon)}\sharp \Big(\Gamma_n^\star.z\cap\overline{B_R}(y)\Big)dy\xrightarrow[]{n\to\infty}0.$$
\end{lemma}
\begin{proof}
    As a consequence of Theorem \ref{geometric plancherel thm}, we only need to check that 
    $$\frac{1}{\vol X_n}\int_{\mathcal{F}_n(\upsilon)}\sharp \Big(\Gamma_{n,\para}.z\cap \overline{B_R}(z)\Big)dy\xrightarrow[]{n\to\infty}0.$$
    We reason with a fixed surface $X=\Gamma\setminus\Hyp$ with $q$ cusps. 
    For any $p\in\{1,\dots, q\}$, let $\gamma_p\in\Gamma$ be the generator of the stabilizer of the cusp $\ca_p$, so that 
    $$\int_{\mathcal{F}(\upsilon)}\sharp \Big(\Gamma_{\para}.z\cap \overline{B_R}(z)\Big)dy=\sum_{p=1}^q\int_{\mathcal{F}(\upsilon)}\sharp\Big(\la \gamma_p\ra.z\cap \overline{B_R}(z)\Big)dz.$$
    For any $z\in \mathcal{F}(\upsilon)$, there exists $d_\upsilon>0$ only depending on $\upsilon$ such that $d(z,\gamma_p.z)\ge d_\upsilon$. One can see this by conjugating $\gamma_p$ to $\left(\begin{smallmatrix}1 & 0\\0 & 1\end{smallmatrix}\right)$ and realizing that $\inf_{\operatorname{Im}(z)\le \upsilon} d(z,z+1)$ is bounded away from $0$. Therefore, for any $p\in\{1,\dots, q\}$ and $y\in \mathcal{F}(\upsilon)$ we have 
    $$\sharp\Big(\la \gamma_p\ra .y\cap \overline{B_R}(y)\Big)\le \frac{\vol B_{R+d_\upsilon/2}}{\vol B_{d_\upsilon/2}}.$$
    Moreover, as a consequence of Lemma \ref{cusps and thin part}, for any $p\in\{1,\dots,q\}$ the integral 
    $$\int_{\mathcal{F}(\upsilon)}\sharp\Big(\la \gamma_p\ra.z\cap \overline{B_R}(z)\Big)dz$$
    is supported on a set of measure $\le C_R$. Summarizing, 

    $$\int_{\mathcal{F_n}(\upsilon)}\sharp\Big(\Gamma_{n,\para}.z\cap \overline{B_R}(z)\Big)dz\le C_r\frac{\vol B_{R+d_\upsilon/2}}{\vol B_{d_\upsilon/2}}q_n.$$
    Since $\frac{q_n}{\vol X_n}\xrightarrow[]{n\to\infty}0$ the proof is complete.
    
\end{proof}

\begin{ex}\label{towers are plancherel}
    Towers of lattices are Plancherel. Let $\Gamma\le \PSL_2(\R)$ be a torsion free lattice and let $X=\Gamma\setminus\Hyp$ be the associated hyperbolic surface. A sequence $(\Gamma_n)_{n\in\N}$ of lattices in is called a \textit{tower of normal subgroups} of $\Gamma$ if:
    \begin{itemize}
        \item $\Gamma_n\trianglelefteq\Gamma$, that is $\Gamma_n$ is a normal subgroup of $\Gamma$, for every $n\in\N$;
        \item $[\Gamma:\Gamma_n]<\infty$;
        \item $\cap_{n\in\N} \Gamma_n=\{\operatorname{id}\}$.
    \end{itemize}
    We now verify that the two conditions of Theorem \ref{geometric plancherel thm} are satisfied, so that the sequence $(X_n)_{n\in\N}$ of hyperbolic surfaces $X_n=\Gamma_n\setminus\Hyp$ is Plancherel.
    
    Condition $1)$ is a direct consequence of the formula for the number of cusps under a normal covering presented in Lemma \ref{lem: n of cusps formula}. In fact, let $q$ be the number of cusps of $X$ and for every $p\in\{1,\dots,q\}$ let $\Gamma_p$ be the stabilizer of a representative for the $p$-th cusp. We have
    $$q_n=[\Gamma:\Gamma_n]\sum_{p=1}^q\frac{1}{[\Gamma_p:\Gamma_p\cap\Gamma_n]}.$$
    Since $\cap_{n\in\N}=\emptyset$, the quantity $[\Gamma_p:\Gamma_p\cap\Gamma_n]$ tends to infinity for any fixed $p$. We conclude by observing that $\vol(X_n)=[\Gamma:\Gamma_n]\vol (X)$.
    
    For condition $2)$, observe that for $n\in \N$ big enough and every $z\in\mathcal{F}_n$ we have $\sharp\big(\Gamma_{n,\hyp}.z\cap\overline{B_R}(z)\Big)=0$. The argument is the same as for the case of compact surfaces (see \cite[Lemma 2.1]{degeorge1978limit}). We summarize it here for the convenience of the reader. On the hyperbolic surface $X$ there are only finitely many closed geodesics with length $\le R$. A closed geodesic in $X$ corresponds to the conjugacy class of a (primitive) hyperbolic element in $\Gamma$. Let $S_R=\{\gamma_1,\dots,\gamma_m\}$ be representatives of such conjugacy classes. since $\cap_{n\in\N}\Gamma_n=\{\operatorname{id\}}$, for $n\in\N$ big enough $\Gamma_n\cap S_R=\emptyset$. Since the $\Gamma_n$'s are normal in $\Gamma$, if $\gamma_i\not\in\Gamma_n$ then no conjugate of $\gamma_i$ belongs to $\Gamma_n$. As a consequence, for $n\in\N$ big enough there are no closed geodesics in $X_n$ of length $\le R$, and in particular $\Gamma_{n,\hyp}.z\cap\overline{B_R}(z)=\emptyset$ for every $z\in\mathcal{F}_n$.

\end{ex}

\subsection{Plancherel convergence and admissible eigenvalue functions}

Our goal now is to extend the results of the previous section to all admissible eigenvalue functions. In order to do this, we recall some results concerning estimates on the number of closed geodesics with bounded length. \newline

\begin{propx}\label{bound on short geodesics}
Let $X$ be a hyperbolic surface of genus $g$ with $q$ cusps. 

\begin{enumerate}[(1)]
    \item If $\beta_1,\dots \beta_k$ are simple closed geodesics of length $\le 2\sinh^{-1}(1)$ on $X$, then $\beta_1,\dots\beta_k$ are disjoint and $k\le 3g-3+q$.
    \item Let $R>0$. Then there are at most $\operatorname{vol}(X)e^{R+6}$ oriented closed geodesics of length $\le R$ which are not iterates of closed geodesics of length $\le 2\sinh^{-1}(1)$.
\end{enumerate}
\end{propx}

\begin{proof}
    (1) follows directly from the Collar Theorem in the non-compact case (\cite[Thm. 4.4.6]{buser2010geometry}). The proof of (2) is analogous to the one for the case of compact surfaces \cite[Thm. 6.6.4]{buser2010geometry}, by using the Collar Theorem in the case of non-compact surfaces. 
\end{proof}

\begin{thm}\label{extended plancherel prop}
     A sequence $(X_n)_{n\in\N}$ of finite-area hyperbolic surfaces is Plancherel if and only if for any admissible eigenfunction $h$
    \begin{equation}\label{extended plancherel}
        \frac{1}{\operatorname{vol}(X_n)}\left(\sum_j h\left(r^{(n)}_j\right)+\frac{1}{4\pi}\int_\R h(r)\frac{-\varphi_{X_n}'}{\varphi_{X_n}}\left(\frac{1}{2}+ir\right)dr\right)\xrightarrow[]{n\to\infty}\mu_{\operatorname{Pl}}(h).
    \end{equation}

\end{thm}
\begin{proof}
  $(\Leftarrow)$ If (\ref{extended plancherel}) is satisfied for any admissible eigenvalue function, then it holds for the eigenvalue function of any kernel $k\in C^\infty_c(\R)_{\operatorname{ev}}$.
    
$(\Rightarrow)$ Let $h$ be an admissible eigenvalue function. Selberg's trace formula is valid for $h$, so that following the same argument as in Proposition \ref{first geometric plancherel interpretation} and Remark \ref{plancherel convergence and closed geodesics}, we see that (\ref{extended plancherel}) holds if and only if 
    $$\frac{1}{\operatorname{vol}(X_n)}\int_{\mathcal{F}_n}\sum_{\Gamma_{n,\hyp}}k(z,\gamma. z)dz\xrightarrow[]{n\to\infty}0$$
    $$\iff \frac{1}{\operatorname{vol}(X_n)}\sum_{\Gamma_{n,\hyp}}\frac{l_{\gamma_0}}{2\sinh(l_\gamma/2)}g(l_\gamma)\xrightarrow[]{n\to\infty}0,$$
    where $k$ is the admissible kernel obtained by $h$ via inverse Selberg transform and $g$ is the Fourier transform of $h$. For each $N\in\N$ we can split the last sum above in two sums, one over geodesics with length $\le N$, one over geodesics with length $>N$:

    $$\frac{1}{\operatorname{vol}(X_n)}\left(\sum_{\substack{\Gamma_{n,\hyp} \\ l_\gamma\le N}}\frac{l_{\gamma_0}}{2\sinh(l_\gamma/2)}g(l_\gamma)+\sum_{\substack{\Gamma_{n,\hyp} \\ l_\gamma>N}}\frac{l_{\gamma_0}}{2\sinh(l_\gamma/2)}g(l_\gamma)\right)\le$$

    $$\left(\underbrace{\frac{||g||_\infty}{\operatorname{vol}(X_n)}\sum_{\substack{\Gamma_{n,\hyp} \\ l_\gamma\le N}}\frac{l_{\gamma_0}}{\sinh\left(\frac{l_\gamma}{2}\right)}}_{(\spadesuit)}\right)+\left(\underbrace{\frac{1}{\operatorname{vol}(X_n)}\sum_{\substack{\Gamma_{n,\hyp}\\ l_\gamma>N}}\frac{l_{\gamma_0}}{\sinh\left(\frac{l_\gamma}{2}\right)}g(l_\gamma)}_{(\heartsuit)}\right).$$

    We suppress the dependence of $(\spadesuit)$ and $(\heartsuit)$ on $n$ and $N$, to avoid the notation becoming too cumbersome.

    Since the sequence $(X_n)_{n\in\N}$ is assumed to be Plancherel, by Lemma \ref{last plancherel lemma} we get
    $$(\spadesuit)\xrightarrow[]{n\to\infty}0.$$
    We turn our attention to $(\heartsuit)$. We split the sum over $l_\gamma>N$ into sums of geodesics with length in the interval $[k,k+1]$ as $n>N$: 
    $$(\heartsuit)=\frac{1}{\operatorname{vol}(X_n)}\sum_{k=N}^\infty\sum_{\substack{\Gamma_{n,\hyp} \\ l_{\gamma\in[k,k+1)}}}\frac{l_{\gamma_0}}{2\sinh(l_\gamma/2)}g(l_\gamma).$$
    In order to estimate $(\heartsuit)$ we need some control on the decay of $g$. It follows from the decay conditions (\ref{admissible eigenvalue functions}) on $h$ that there exists $0<\varepsilon<1$ such that $g(t)\le C_\varepsilon e^{\big(-\frac{1}{2}-\varepsilon\big)|t|}$ for some positive constant $C_\varepsilon$. We use Theorem \ref{bound on short geodesics} to obtain the following upper bound for $(\heartsuit)$:

   $$\underbrace{C_\varepsilon\sum_{k=N}^\infty \frac{(k+1)e^{k+7}e^{\big(-\frac{1}{2}-\varepsilon\big)k}}{2\sinh(k/2)}}_{(I)}
        +\underbrace{\frac{1}{\operatorname{vol}(X_n)}\sum_{k=N}^\infty\sum_{\substack{\Gamma_{n,\hyp} \\ l_{\gamma_0}\le2\sinh^{-1}(1) \\ l_\gamma\in[k,k+1)}}\frac{l_{\gamma_0}}{2\sinh(l_\gamma/2)} g(l_\gamma)}_{(II)}.$$

We study the quantities $(I)$ and $(II)$ separately. 
$$(I)=C_\varepsilon e^7\sum_{k=N}^\infty \frac{(k+1)e^{\frac{k}{2}-\varepsilon k}}{e^{\frac{k}{2}}(1-e^{-k})}\le\frac{C_\varepsilon e^7}{1-e^{-1}}\sum_{k=N}^{\infty}(k+1)e^{-\varepsilon k}.$$

The series $\sum_{k=0}^{\infty}(k+1)e^{-\varepsilon k}$ converges, so that the tail $\sum_{k=N}^{\infty}(k+1)e^{-\varepsilon k}$ converges to $0$ as $N\to\infty$. In particular, 

$$(I)\xrightarrow[]{N\to\infty}0.$$

Concerning $(II)$, let $\{\gamma_1,\dots \gamma_{s(n)}\}$ be the collection of closed geodesics on $X_n$ with length $\le 2\sinh^{-1}(1)$. From Theorem \ref{bound on short geodesics}, we know $s(n)\le 3(g_n-1)+q_n$. Also, notice that for $i=1,\dots s(n)$ the quantity $\frac{1}{l_{\gamma_i}}$ bounds the number of iterates of $\gamma_i$ in any interval $[k,k+1)$. Hence 
\begin{equation*}
    \begin{split}
        (II)\le&\frac{1}{\text{vol}(X_n)}\sum_{k=N}^{\infty}\sum_{i=1}^{s(n)}\frac{1}{l_{\gamma_i}}\frac{l(\gamma_i)}{2\sinh(k/2)}g(k) \\
    \le &\frac{1}{\text{vol}(X_n)}\frac{1}{2\sinh(N/2)}\sum_{k=N}^{\infty}\sum_{i=1}^{s(n)}e^{\big(\frac{1}{2}-\varepsilon\big)k}\\
    \le &\frac{3g_n-3+q_n}{\text{vol}(X_n)}\frac{1}{2\sinh(N/2)}\sum_{k=N}^{\infty}e^{\big(\frac{1}{2}-\varepsilon\big)k} \\
    \le & \frac{1}{2\sinh(N/2)}\sum_{k=N}^{\infty}e^{\big(\frac{1}{2}-\varepsilon\big)k}\\
    =& O\left(\frac{1}{\sinh(N/2)}\right).
    \end{split}
\end{equation*}

Summarizing, for every $\delta>0$ there exists $N$ large enough, depending only on the admissible eigenvalue function $h$, such that 
$$(\heartsuit)<\delta \qquad \text{for every }n\in\N.$$

Therefore, 
$$\lim_{n\to\infty}\frac{1}{\operatorname{vol}(X_n)}\int_{\mathcal{F}_n}\sum_{\Gamma_{n,\hyp}} k(z,\gamma .z)dz<\delta\qquad \text{for every }\delta>0.$$
Letting $\delta\to0$ concludes the proof.

\end{proof}

\subsection{Asymptotic behavior of the Laplace spectrum}

Let $X$ be a hyperbolic surface with finite area. For a compact interval $I\subset[0,\infty)$, we denote by $N(X,I)$ the number of Laplace eigenvalues (from the discrete spectrum) on $X$ in the interval $I$, counted with multiplicities. Moreover, we write 
$$M(X,I)=\frac{1}{4\pi}\int_{\tau^{-1}(I)}\frac{-\varphi_{X}'}{\varphi_{X}}\left(\frac{1}{2}+is\right)ds.$$
 The sum $N(X,I)+M(X,I)$ measures the contribution of the discrete and continuous spectrum in the interval $I$. 
 
% The problem of understanding the asymptotic behavior of the Laplace spectrum has been extensively studied in relation to equidistribution of $L^2$-eigenfunctions and Eisenstein series.
For $T>0$, we write $I_T:=[\frac{1}{4},\frac{1}{4}+T^2]$. Weyl's asymptotic law (see e.g. \cite[Chap. 11]{iwaniec2021spectral}) asserts that
$$N(X,I_T)+M(X,I_T)=\frac{\vol (X)}{4\pi}T^2+O(T\log T)\qquad \text{as }T\to\infty.$$
On the other hand, if one wants to study the contribution of the Laplace spectrum in a fixed interval $I$, the idea is to consider a sequence of surfaces $(X_n)_{n\in\N}$ and to study $\frac{N(X_n,I)+M(X_n,I)}{\operatorname{vol}(X_n)}$ as $n\to\infty$. In the latter case, asking that $\operatorname{vol}(X_n)\to\infty$ does not seem to be enough to obtain meaningful asymptotics. Le Masson and Sahlsten have shown that, under the assumptions that a sequence is Benjamini--Schramm and there is a uniform lower bound on the systoles, it holds $N(X_n,I)+M(X_n,I)\sim\operatorname{vol}(X_n)$ (see \cite[Lemma 9.1]{le2017quantum} for a proof in the compact case and \cite[Thm. 5.1]{le2024quantum} for a proof in the non-compact case). The aim of this section is to show that the same asymptotic as in \cite{le2024quantum} can be obtained under the assumption that a sequence is Plancherel, dropping the assumption on the uniform bound on the systoles. In section \ref{bs and plancherel} we shall discuss precisely how Plancherel convergence relates to Benjamini--Schramm convergence.\newline

\begin{propx}[spectral convergence]\label{prop: spectral convergence}
    Let $(X_n)_{n\in\N}$ be a Plancherel sequence of finite area hyperbolic surfaces.
    Then for any compact interval $I\subset[0,\infty)$ we have  
    $$\frac{N(X_n,I)+M(X_n,I)}{\operatorname{vol}(X_n)}\xrightarrow[]{n\to\infty}\mu_{\operatorname{Pl}}(I),$$
    where $\mu_{\operatorname{Pl}}(I)=\frac{1}{4\pi}\int_\R \1_I\left(\frac{1}{4}+s^2\right)\tanh(\pi s)sds.$
\end{propx}
\begin{proof}

In \cite[Theorem 5.1]{le2024quantum}, the uniform lower bound on systoles is only used to prove that for each $t>0$
\begin{equation*}
    \begin{split}
        \lim_{n\to\infty}&\frac{1}{\operatorname{vol}(X_n)}\left(\sum_j e^{-t\lambda_j^{(n)}}+\frac{1}{4\pi}\int_\R \frac{-\varphi'_{X_n}}{\varphi_{X_n}}\left(\frac{1}{2}+is\right)e^{-t(1/4+s^2)}ds\right)\\
        =&\frac{1}{4\pi}\int_\R e^{-t(1/4+s^2)}\tanh(\pi s)sds.
    \end{split}
\end{equation*}
Since the sequence $(X_n)_{n\in\N}$ is Plancherel and $h_t(r)=e^{-t(1/4+r^2)}$ is admissible (it is the eigenvalue function of the heat kernel), we can conclude that the above limit holds thanks to Proposition \ref{extended plancherel prop}. For this, we do not need the assumption of a uniform lower bound on the systoles. For the rest of the proof, one can reproduce verbatim the argument in \cite[Theorem 5.1]{le2017quantum}.

\end{proof}

\subsection{Comparison with Benjamini--Schramm convergence}

\subsubsection*{Benjamini--Schramm convergence}

Let $X=\Gamma\setminus\Hyp$ be a hyperbolic surface with finite area. For $p\in X$, we define the \textbf{injectivity radius} at $p$ as 
$$\operatorname{inj}_X(p)=\frac{1}{2}\inf\{d(z,\gamma .z)\ :\ \gamma\in\Gamma\smallsetminus\{\text{id}\}\},$$
where $z$ is any lift of $p$ to $\Hyp$. For $R>0$, we define the $R$-\textit{thin part} of $X$ as 
$$X_{<R}=\{ p\in X\ :\ \operatorname{inj}_X(p)<R\}.$$

We say that a sequence $(X_n)_{n\in\N}$ of finite-area hyperbolic surfaces is \textbf{Benjamini--Schramm convergent} to $\Hyp$ if for any $R>0$
$$\frac{\operatorname{vol}((X_n)_{<R})}{\operatorname{vol}(X_n)}\xrightarrow[]{n\to\infty}0.$$

\begin{propx}\label{bs characterization}
     A sequence $(X_n)_{n\in\N}$ of finite-area hyperbolic surfaces is Benjamini--Schramm convergent if and only if

      \begin{itemize}
        \item $\frac{q_n}{\operatorname{vol}(X_n)}\xrightarrow{n\to\infty}0$, and
        \item for all $R>0$
         \begin{equation}\label{geometric BS}
        \frac{1}{\operatorname{vol}(X_n)}\int_{\F_n}\operatorname{sign}\left(\sharp\left(\Gamma_{n,\hyp}.z\cap \overline{B_R}(z)\right)\right)dz\xrightarrow[]{n\to\infty}0,
    \end{equation}
    \end{itemize}
\end{propx}
\begin{proof}
     $(\Rightarrow)$ Let $\{\ca^{(n)}_1,\dots,\ca_q^{(n)}\}$ be the the set of cusps of $X_n$. According to \cite[section 4.4]{buser2010geometry}, to any cusp corresponds a cuspidal zone $\mathscr{C}^*$ in $X_n$ isometric to $(-\infty,\log 2]\times\mathcal{S^1}$, with metric $d\rho^2+e^{2\rho}dt^2$. Moreover, a point in such cuspidal zone has injectivity radius $\le\sinh^{-1}(1)$. The fact that $\frac{q_n}{\operatorname{vol}(X_n)}\xrightarrow[]{n\to\infty}0$ is then immediate, for otherwise 
     $$\frac{\operatorname{vol}((X_n)_{<R})}{\operatorname{vol}(X_n)}\ge \operatorname{vol}(\mathscr{C}^*)\frac{q_n}{\operatorname{vol}(X_n)}\not\xrightarrow[]{}0$$
     for any $R>\sinh^{-1}(1).$ Now notice that Benjamini--Schramm convergence for the sequence $(X_n)_{n\in\N}$ can be reformulated as 
     $$\frac{1}{\operatorname{vol}(X_n)}\int_{\F_n}\operatorname{sign}\left(\sharp\left(\Gamma_n^\star.z\cap \overline{B_R}(z)\right)\right)dz\xrightarrow[]{n\to\infty}0.$$
     Therefore, 
     $$\int_{\F_n}\operatorname{sign}\left(\sharp\left(\Gamma_{n,\hyp}.z\cap \overline{B_R}(z)\right)\right)dz\xrightarrow[]{n\to\infty}0,$$
     as in the latter integral we are considering less lattice points. 

     $(\Leftarrow)$ For each $k=1,\dots,q$, let $\gamma_k^{(n)}\in\Gamma_n$ be as defined in section \ref{cusps}. Notice that
     $$\int_{\F_n}\operatorname{sign}\left(\sharp\left(\Gamma_n^\star.z\cap \overline{B_R}(z)\right)\right)dz\le$$
     $$\sum_{k=1}^{q_n}\int_{F_n}\sharp\left(\gamma_k^{(n)}(z)\cap \overline{B_R}(z)\right)dz+\int_{\F_n}\operatorname{sign}\left(\sharp\left(\Gamma_{n,\hyp}.z\cap \overline{B_R}(z)\right)\right)dz.$$
     Therefore,
     \begin{equation*}
         \begin{split}
             &\lim_{n\to\infty}\frac{1}{\operatorname{vol}(X_n)}\int_{\F_n}\operatorname{sign}\left(\sharp\left(\Gamma_n^\star.z\cap \overline{B_R}(z)\right)\right)dz\le\\
             &\lim_{n\to\infty}\frac{1}{\operatorname{vol}(X_n)}\sum_{k=1}^{q_n}\int_{F_n}\sharp\left(\gamma_k^{(n)}(z)\cap \overline{B_R}(z)\right)dz.
         \end{split}
     \end{equation*}
    Next, by Lemma \ref{cusps and thin part} there exists a constant $C_R$ such that for each $k=1,\dots,q_n$ 
    $$\int_{F_n}\sharp\left(\gamma^{(n)}_k(z)\cap \overline{B_R}(z)\right)dz\le C_R.$$
    Therefore,
     \begin{equation*}
         \begin{split}
             &\lim_{n\to\infty}\frac{1}{\operatorname{vol}(X_n)}\int_{\F_n}\operatorname{sign}\left(\sharp\left((\Gamma_n\smallsetminus\{\text{Id}\})z\cap \overline{B_R}(z)\right)\right)dz\le\\
             &\lim_{n\to\infty}\frac{1}{\operatorname{vol}(X_n)}\sum_{k=1}^{q_n}C_R=\lim_{n\to\infty}\frac{q_n}{\operatorname{vol}(X_n)}C_R=0.  \qquad\qquad \qedhere
         \end{split}       
     \end{equation*}
   
\end{proof}

\subsubsection*{Benjamini--Schramm and Plancherel convergence}\label{bs and plancherel}

\begin{propx}\label{PL vs BS}
    Let $(X_n)_{n\in\N}$ be a sequence of finite-area hyperbolic surfaces. Then 
    \begin{enumerate}[a)]
        \item If $(X_n)_{n\in\N}$ is Plancherel convergent, then $(X_n)_{n\in\N}$ is Benjamini--Schramm convergent to $\Hyp$.
        \item If $(X_n)_{n\in\N}$ is Benjamini--Schramm convergent to $\Hyp$ and there is a uniform lower bound on the systole of $X_n$, then $(X_n)_{n\in\N}$ is Plancherel convergent.
    \end{enumerate}
\end{propx}
\begin{proof}
    (a) Since $(X_n)_{n\in\N}$ is Plancherel, then $\frac{q_n}{\operatorname{vol}(X_n)}\xrightarrow[]{n\to\infty}0$. Moreover, for each $n\in\N$
    $$\frac{1}{\operatorname{vol}(X_n)}\int_{\F_n} \operatorname{sign}\left(\sharp\left(\Gamma_{n,\hyp}.z\cap \overline{B_R}(z)\right)\right)dz\le\frac{1}{\operatorname{vol}(X_n)}\int_{\F_n} \sharp\left(\Gamma_{n,\hyp}.z\cap \overline{B_R}(z)\right)dz.$$
    By Theorem \ref{geometric plancherel thm}, the latter quantity tends to $0$ as $n\to\infty$. By Proposition \ref{bs characterization}, the sequence is Benjamini--Schramm convergent to $\Hyp$.

    (b) Since $(X_n)_{n\in\N}$ is Benjamini--Schramm convergent to $\Hyp$, then $\frac{q_n}{\vol(X_n)}$ tends to 0 as $n\to\infty$. Let $\eps>0$ such that $\operatorname{sys}(X_n)>\eps$ for each $n\in\N$. Then 
    $$\frac{1}{\operatorname{vol}(X_n)}\int_{\F_n} \sharp\left(\Gamma_{n,\hyp}.z\cap \overline{B_R}(z)\right)dz\le\frac{\eps}{\operatorname{vol}(X_n)}\int_{\F_n} \operatorname{sign}\left(\sharp\left(\Gamma_{n,\hyp}.z\cap \overline{B_R}(z)\right)\right)dz.$$
    By Proposition \ref{bs characterization}, the latter quantity tends to $0$ as $n\to\infty$. By Theorem \ref{geometric plancherel thm}, the sequence is Benjamini--Schramm convergent.
\end{proof}

\begin{rem}
    There exist Plancherel sequences with no uniform lower bound on the systoles and Benjamini--Schramm sequences which are not Plancherel sequences. Examples can be constructed as in \cite{gavelli2025benjamini}, by keeping some pairs of cusps for each surface.
\end{rem}

\section{Quantum Ergodicity and Automorphic Kernels}\label{quantum ergodicity and smooth kernels}

This section is dedicated to the proof of the Quantum Ergodicity Theorem.

\subsection{Notation}\label{setting}

Let $\Gamma\le G$ be a torsion free lattice, let $X$ be the associated hyperbolic surface and let $\cu(X)=\{\ca_1,\dots,\ca_q\}$ be the set of cusps of $X$. For $s\in\R$ and $p\in\{1,\dots, q\}$ we write
    $$E_p(s):= E_{\ca_p}\left(\cdot,\frac{1}{2}+is\right).$$

Let $k:\Hyp\times\Hyp\to\C$ be a bounded measurable kernel with the following properties:
\begin{align*}
    &1)\ k(\gamma.x,\gamma.y)=k(x,y) \text{  for every } \gamma\in\Gamma\qquad &(\Gamma\text{-invariance})\\
    &2)\ \exists\ M>0\ s.t.\ k(x,y)=0 \text{ if } d(x,y)> M &(\text{finite propagation)}  
\end{align*}

Such a kernel gives rise to an operator $T:C(\Hyp)\to C(\Hyp)$ by convolution

$$Tf(x):=\int_\Hyp k(x,y)f(y) dy$$

that descends by $\Gamma$-invariance to an operator on $C(X)$: if $g\in C(X)$, 
$$Tg(x)=\int_X \sum_{\gamma\in\Gamma} k(x,\gamma .y) g(y)dy.$$

For $s\in[0,\infty)$, we denote 
$$\la T\ra _{s}:=\frac{1}{\vol(X)}\int_\mathcal{F}\int_\Hyp k(x,y)\phi_s(d(x,y))dxdy,$$
where $\mathcal{F}$ is a fundamental domain for $X$ in $\Hyp$.

We will later see in section \ref{radialization of kernels} that this quantity corresponds to the spherical transform of a radialization of the kernel $K$ evaluated at the spectral parameter $s$. The next step is to define the \textbf{quantum mean absolute deviation} over $I$ of the Laplace eigenfunctions with respect to the operator $T$. First we introduce some more notation: 

$$\mathcal{D}_{X,I}(T):=\sum_{j\ \colon \lambda_j\in I}|\la T\psi_j,\psi_j\ra-\la T \ra_{s_j}|,$$
$$\mathcal{C}_{X,I}(T):=\frac{1}{4\pi}\int_{\tau^{-1}(I)}\left|\sum_{p=1}^q\la TE_p(s),E_p(s)\ra +\frac{\varphi_X'}{\varphi_X}\left(\frac{1}{2}+is\right)\la T\ra_{s}\right| ds,$$

where $\lambda_j=\frac{1}{4}+s_j^2$. The contribution of the discrete and the continuous spectra is, respectively
$$N(X,I):= \sharp\{j\ \colon \lambda_j\in I\},$$
$$M(X,I):=\frac{1}{4\pi}\int_{\tau^{-1}(I)}\frac{-\varphi_X'}{\varphi_X}\left(\frac{1}{2}+is\right)ds.$$

The \textbf{quantum mean absolute deviation} over $I$ of the Laplace eigenfunctions with respect to the operator $T$ is 
$$\operatorname{Dev}_{X,I}(T):=\frac{\mathcal{D}_{X,I}(T)+\mathcal{C}_{X,I}(T)}{N(X,I)+M(X,I)}.$$

\begin{defn}\label{uniform relative compact support}

Let $\pi:\Hyp\to X$ be the quotient projection. For $\upsilon\ge 1$, recall the decomposition of $X$ as a disjoint union of a compact surface and cuspidal zones

$$X=X(\upsilon)\cup\bigcup_{p=1}^qX_p(\upsilon).$$

We say that a $\Gamma$-invariant kernel $k:\Hyp\times\Hyp\to\C$ has \textbf{relative compact support} (or compact support relative to the cusp decomposition) if there exists $\upsilon \ge 1$ such that $k(x,y)=0$ whenever $x\not\in \pi^{-1}X(\upsilon)$ or $y\not\in \pi^{-1}X(\upsilon)$. In this case we also say that $k$ has relative compact support in $X(\upsilon)$. Notice that $k$ having relative compact support in $X(\upsilon)$ is equivalent to requiring that the automorphic kernel $\textbf{k}(x,y):=\sum_{\gamma\in\Gamma} k(x,\gamma .y)$ on $X\times X$ is such that $\textbf{k}(x,y)=0$ whenever $x\not\in X(\upsilon)$ or $y\not\in X(\upsilon)$.

If $(\Gamma_n)_{n\in\N}$ is a sequence of torsion free lattices and $(k_n)_{n\in\N}$ is a sequence of $\Gamma_n$-invariant kernels, we say that the latter has \textbf{uniformly relative compact support} if there exists $\upsilon\ge1$ such that for every $n\in\N$ the kernel $k_n$ has relative compact support in $X_n(\upsilon)$.

\end{defn}

Now that all the necessary notions have been introduced, we state again the Quantum Ergodicity Theorem (and number it for further reference). 

\begin{thm}\label{QE theorem}
    Let $I\subset (1/4,\infty)$ be a compact interval and $(X_n)_{n\in\N}$ a sequence of finite area hyperbolic surfaces. Denote by $q_n$ the number of cusps of $X_n$. Assume that 
    \begin{enumerate}
        \item $(X_n)_{n\in\N}$ is a Plancherel sequence;
        \item there is a uniform spectral gap for the Laplacian on $X_n$;
        \item $\frac{q_n^2}{\vol (X_n)}\xrightarrow[]{n\to\infty}0$.
    \end{enumerate}
        
     Then, for any  sequence of uniformly bounded measurable kernels $(k_n)_{n\in\N}$ with uniformly relative compact support and uniformly finite propagation,

    $$\operatorname{Dev}_{X_n,I}(T_n)\xrightarrow[]{n\to\infty}0.$$
\end{thm}

This Quantum Ergodicity Theorem can be interpreted as follows: the two points correlation function of the eigenfunctions converges, on average, to the two-point correlation function of the spherical function with the same spectral parameter.

\subsection{Radialization of kernels}\label{radialization of kernels}

% A smooth kernel $k:\Hyp\times\Hyp\to\C$ is said to be a \textbf{radial kernel} if it only depends on the distance between two points, that is there exists a smooth even function $u:[0,\infty)\to\C$ such that $k(x,y)=u(d(x,y))$ for every $x,y\in\Hyp$. 

For this section, we fix a torsion free lattice $\Gamma\le G$ and a measurable $\Gamma$-invariant kernel $k$ with finite propagation. Following \cite{abert2022eigenfunctions}, we define a function $[k]:[0,\infty)\to\C$ by taking radial averages of $k$:

$$[k](r)=\fint_X\fint_{\partial B_r(x)}k(x,y)dA(y)dx,$$

where $dA$ is the area measure on the boundary, which for $r>0$ is the $1$-dimensional Hausdorff measure induced by the restriction of the hyperbolic metric to the boundary. For $r=0$ it is the Dirac measure supported on $\{x\}$. 

In polar coordinates based at $x$, we can reformulate the above integral as 

% $$[k](r)=\fint_{\Gamma\setminus \Hyp}\int_K K(x,xka_r)d_k dx=\fint_{\Gamma\setminus G}K(x,xa_r)dx$$

% or equivalently as

$$[k](r):=\fint_\mathcal{F}\fint_{\mathbb{S}^1} k(x,\exp_x(r,\theta))d\theta dx.$$
Here $\mathcal{F}$ is a fundamental domain of $X$ in $\Hyp$, 

The function $[k]$ induces a $G$-invariant kernel on $\Hyp\times\Hyp$, which by abuse of notation we still denote $[k]$, by $[k](x,y):=[k](d(x,y))$. 
We denote the operator by convolution with the radial kernel $[k]$ by $[T]:C(X)\to C(X)$,

$$[T]f(x):=\int_\Hyp [k](x,y)f(y) dy.$$

Notice that the radialization procedure we performed above is idempotent, that is $[[k]]=[k]$, in fact 

\begin{align*}
    [[k]](r)=&\fint_X\fint_{\partial B_r(x)}[k](d(x,y))dA(y)dx \\
     =&\fint_X\fint_{\partial B_r(x)} [k](r) dA(y)ydx \\
     =& [k](r)\fint_X\fint_{\partial B_r(x)} dA(y)ydx\\
     =&[k](r).
\end{align*}

Therefore, $[[T]]=[T]$. 

Recall that any Laplace eigenfunction on $X$ of eigenvalue $\lambda=\frac{1}{4}+s^2$ is an eigenfunction of any operator by convolution with a kernel $k\in L^\infty(\R_+)$ with compact support, with eigenvalue given by the spherical transform $\hat k$ evaluated at $s$ (Remark \ref{spherical transform for non smooth kernels}).

\begin{lemma}\label{properties of radial kernel}
    If $\psi_j$ is a Laplacian eigenfunction with eigenvalue $\lambda_j=\frac{1}{4}+s_j^2$ normalized to be of norm one, then
    $$\la [T]\psi_j,\psi_j\ra=\widehat{[k]}(s_j).$$
    Moreover, for any $s\in[0,\infty)$, 
    $$\widehat{[k]}(s)=\la T\ra _s.$$
\end{lemma}
\begin{proof} Let $\psi_j$ be an eigenfunction with eigenvalue $\lambda_j=\frac{1}{4}+s_j^2$ normalized to be of norm one. Then
    \begin{equation}\label{spherical transform and eigenfunctions}
\begin{split}
    \la [T]\psi_j,\psi_j\ra&=\int_X([T]\psi_j)(x)\overline{\psi_j}(x)dx \\
    &=\int_X\widehat{[k]}(s_j)|\psi_j(x)|^2dx \\
    &=\widehat{[k]}(s_j) ||\psi_j||_2^2 \\
    &= \widehat{[k]}(s_j). 
\end{split}
\end{equation}

On the other hand, 

\begin{equation}\label{spherical transform and average}
\begin{split}
    \widehat{[k]}(s)&=\int_0^\infty [k](r)\phi_s(r)\sinh(r) dr\\
    &=\int_0^\infty\left(\fint_\mathcal{F}\fint_{\mathbb{S}^1} k(x,\exp_x(r,\theta) d\theta dx\right)\phi_s(r)\sinh(r) dr\\
    &=\fint_\mathcal{F}\left(\int_0^\infty\fint_{\mathbb{S}^1} k(x,\exp_x(r,\theta))\phi_{s}(r)\sinh(r)d\theta dr\right) dx\\
    &=\fint_\mathcal{F}\int_\Hyp k(x,y)\phi_s(d(x,y))dy dx\\
    &=\la T \ra_s.
    \end{split}
\end{equation}

In the penultimate step we consider $f_x(r,\theta):=k(x,\exp_x(r,\theta))$ and use formula (\ref{eq: polar coordinates}) for the integral in polar coordinates based at $x$. 
\end{proof}

Next we show that the $L^\infty$-norm of the spherical transform of the radialized kernel $[k]$ is controlled by the $L^\infty$-norm of the kernel $k$ itself.

\begin{lemma}\label{bound on infty norm}

    $$\Big|\Big|\widehat{[k]}\Big|\Big|_{L^{\infty}([0,\infty)}\le \sinh(M)\big|\big|k\big|\big|_{L^\infty(X\times X)}.$$

\end{lemma}
\begin{proof}
    Notice that $[k](r)=0$ whenever $r\ge M$ and that
    $$\Big|[k](r)\Big|\le\fint_X\fint_{B_r(x)}\Big| k(x,y)\Big| dxdy\le \big|\big| k\big|\big|_{L^{\infty}(X\times X)}.$$
    Therefore,
    $$\Big|\widehat{[k]}(s)\Big|\le\int_0^M \Big|[k](r)\phi_s(r)\sinh(r)\Big|dr\le\big|\big|k\big|\big|_{L^\infty(X\times X)}\sinh(M).$$
    The result follows by noticing that the RHS does not depend on $s$. We also used that the spherical functions satisfy the bound $|\phi_s(r)|\le 1$ for all $r,s\ge 0$.
\end{proof}

\subsection{Reduction to radial average zero operators}

The goal of this section is to show that it is enough to prove Theorem \ref{QE theorem} for operators such that $\la T_n \ra_s=0$ for every $s\in [0,\infty)$ (see Proposition \ref{average zero proposition}). As a consequence of  (\ref{spherical transform and average}), this is equivalent to requiring that the kernels $k_n$ are such that $\widehat{[k_n]}=0$, which in turn is equivalent (by injectivity of the spherical transform) to $[k_n]=0$. We call a kernel $k$ with the property that $[k]=0$ a \textbf{radial average zero kernel}. Analogously, an operator with integral kernel a radial average zero kernel is called a \textbf{radial average zero operator}. This step is analogous to the reduction to mean zero test functions in \cite{le2024quantum}. In the following, we fix a compact interval $I\subset(\frac{1}{4},\infty)$, a $\Gamma$-invariant measurable kernel $k:\Hyp\times\Hyp\to\C$ with finite propagation and we let $T$ be the operator on $C(X)$ induced by $k$. For $\upsilon\ge1$, define $\chi=\chi_\upsilon$ by 
\begin{equation}\label{chi}
\chi(x):=\begin{cases}
        \frac{\operatorname{vol}(X)}{\vol(X(\upsilon))},\qquad &x\in X(\upsilon)\\
        0, & \text{otherwise}
    \end{cases}
    \end{equation}
and 
$$T_\chi:=T-\chi[T],$$
where $(\chi[T])f(x):=\chi(x)[T]f(x)$. Moreover, by abuse of notation\footnote{$\chi$ is not an operator with a measurable kernel. However, the quantities $\mathcal{D}_{X,I}(\chi)$ and  $\mathcal{C}_{X,I}(\chi)$ are, respectively, the contribution of $L^2$-eigenfunctions and Eisenstein series to the absolute mean absolute deviation of $X$ over $I$ with respect to the operator of multiplication by $\chi$, which is the case dealt with in \cite{le2024quantum}. Notice that $\chi$ has mean $\la \chi\ra=1$ over $X$.}, we write 
$$\mathcal{D}_{X,I}(\chi):=\sum_{j \colon \lambda_j\in I} |\la \chi\psi_j,\psi_j\ra -1|,$$
$$\mathcal{C}_{X,I}(\chi):=\frac{1}{4\pi}\int_{\tau^{-1}(I)}\left|\sum_{p=1}^q\la \chi E_p(s),E_p(s)\ra+\frac{\varphi_X'}{\varphi_X}\left(\frac{1}{2}+is\right)\right|ds.$$

Moreover, notice that the kernel of $T_\chi$ is     $$k_\chi(x,y)=k(x,y)-\chi(x)[k](x,y),$$
so that $[k_\chi]=0$ and 
\begin{equation}\label{zero average auxiliary operator}
    \la T_\chi\ra _s=\widehat{[k_\chi]}(s)=0 
\end{equation}
for every $r\in[0,\infty)$.

\begin{lemma}\label{averaging discrete and continuous contribution}
    Let $X=\Gamma\setminus\Hyp$ be a hyperbolic surface with finite area, consider a measurable $\Gamma$-invariant kernel $k:\Hyp\times \Hyp\to\C$ with finite propagation $M$ and let $T:C(X)\to C(X)$ be the operator by convolution with $k$. Fix $t\ge0$ and let $\chi=\chi_t$ as in (\ref{chi}). Then  
    $$\mathcal{D}_{X,I}(T)\le\mathcal{D}_{X,I}(T_\chi)+\sinh(M)||k||_\infty\mathcal{D}_{X,I}(\chi)$$
    and 
    $$\mathcal{C}_{X,I}(T)\le\mathcal{C}_{X,I}(T_\chi')+\sinh(M)||k||_{\infty}\mathcal{C}_{X,I}(\chi),$$
\end{lemma}
\begin{proof}
    If $\psi_j$ is a Laplace eigenfunction with eigenvalue $\lambda_j=\frac{1}{4}+s_j^2$ of $L^2$-norm $1$,
    \begin{align*}
        \la T\psi_j,\psi_j\ra-\la T\ra_{s_j}&=\la (T_\chi+\chi[T])\psi_j,\psi_j\ra-\widehat{[k]}(s_j)\\
        &=\la T_\chi\psi_j,\psi_j\ra +\widehat{[k]}(s_j)\Big(\la \chi\psi_j,\psi_j\ra-1\Big).
    \end{align*}

    Hence,
    \begin{align*}
        \mathcal{D}_{X,I}(T)&=\sum_{j\ \colon \lambda_j\in I}\bigg|\la T\psi_j,\psi_j\ra -\la T\ra_{s_j}\bigg|\\
        &\le\sum_{j\ \colon \lambda_j\in I} \bigg|\la T_\chi\psi_j,\psi_j\ra \bigg| + \Big|\Big|\widehat{[k]}\Big|\Big|_\infty\sum_{j\ \colon \lambda_j\in I} \bigg|\la \chi\psi_j,\psi_j\ra-1\bigg| \\
        &\le\mathcal{D}_{X,I}(T')+\sinh(M)||k||_\infty\mathcal{D}_{X,I}(\chi).
    \end{align*}
Where we used that $T_\chi$ is a radial average zero operator (\ref{zero average auxiliary operator}) and the bound in the last step is the content of Lemma \ref{bound on infty norm}.
We proceed similarly for the continuous part:

\begin{align*}
    &\mathcal{C}_{X,I}(T)=\mathcal{C}_{X,I}(T_\chi+\chi[T]) \\
    &=\frac{1}{4\pi}\int_{\tau^{-1}(I)}\left|\sum_{p=1}^q\la (T_\chi+\chi[T])E_p(s),E_p(s)\ra+\frac{\varphi_X'}{\varphi_X}\left(\frac{1}{2}+is\right)\la T\ra_s\right|ds\\
    &\le \mathcal{C}_{X,I}(T_\chi)+\frac{1}{4\pi}\int_{\tau^{-1}(I)}\left|\sum_{p=1}^q\la \chi[T] E_p(s),E_p(s)\ra +\frac{\varphi_X'}{\varphi_X}\left(\frac{1}{2}+is\right)\la T\ra_s\right|ds\\
    &=\mathcal{C}_{X,I}(T_\chi)+\frac{1}{4\pi}\int_{\tau^{-1}(I)}\Big|\widehat{[k]}(s)\Big|\left|\sum_{p=1}^q\la \chi E_p(s),E_p(s)\ra +\frac{\varphi_X'}{\varphi_X}\left(\frac{1}{2}+is\right)\right| ds\\
    &\le \mathcal{C}_{X,I}(T_\chi)+\Big|\Big|\widehat{[k]}\Big|\Big|_\infty\mathcal{C}_{X,I}(\chi)\\
    &\le \mathcal{C}_{X,I}(T_\chi)+\sinh(M)||k||_\infty\mathcal{C}_{X,I}(\chi).
\end{align*}

Where we used that $T_\chi$ is a radial average zero operator (\ref{zero average auxiliary operator}) and the bound in the last step is the content of Lemma \ref{bound on infty norm}.
\end{proof}

\begin{propx}\label{average zero proposition}
    Assume that Theorem \ref{QE theorem} holds for any sequence of uniformly bounded $\Gamma_n$-invariant kernels $(k_n)_{n\in\N}$ with uniformly relative compact support, uniformly finite propagation and radial average $0$. Then it holds for any sequence of uniformly bounded $\Gamma_n$-invariant kernels with uniformly relative compact support and uniformly finite propagation.
\end{propx}
\begin{proof}

    For $n\in\N$ we denote $\pi_n:\Hyp\to X_n$ the quotient projection. 

    By assumption there exists $\upsilon\ge 1$ such that for every $n\in\N$ the kernel $k_n$ has relative compact support in $X_n(\upsilon)$.
    
    % The sequence of kernels has uniform relative compact support if and only if there exist:
    % \begin{itemize}
    %     \item $t>0$ such that $k_n(x,y)=0$ whenever $x\notin X_n(t)$ or $y\notin X_n(t)$;
    %     \item 
    % \end{itemize}
    We define $\chi_n$ as in (\ref{chi}) with respect to such uniform $
\upsilon$. The sequence of kernels having uniformly finite propagation means that there exists $M>0$ such that 
    $$k_n(x,y)=0\qquad \text{whenever } d_{X_n}(x,y)\ge M.$$ 
    Moreover, since the sequence of kernels is uniformly bounded, there exists $L>0$ such that $||k_n||_\infty<L$ for every $n\ge 0$. 

    As a consequence of Lemma \ref{averaging discrete and continuous contribution} , 
    $$\operatorname{Dev}_{X_n,I}(T_n)\le \operatorname{Dev}_{X_n,I}(T_{\chi_n})+L\sinh(M)\operatorname{Dev}_{X_n,I}(\chi_n).$$

    Now, $T_{\chi_n}$ is a radial average zero operator for any $n\in\N$ (\ref{zero average auxiliary operator}). Therefore, by assumption, 
    $$\operatorname{Dev}_{X_n,I}(T_{\chi_n})\xrightarrow[]{n\to\infty}0.$$
    On the other hand, it is the content of \cite[Theorem 1.2]{le2024quantum} that 
    $$\operatorname{Dev}_{X_n,I}(\chi_n)\xrightarrow[]{n\to\infty}0.$$

    We point out that Theorem 1.2 in \cite{le2024quantum} is proven under the assumption that the sequence $(X_n)_{n\in\N}$ is Benjamini-Schramm convergent to $\Hyp$ with a uniform lower bound on the systole, whereas we only assume the sequence to be Plancherel. First off, a Plancherel sequence is Benjamini-Schramm convergent to $\Hyp$ (Prop. \ref{PL vs BS}). Secondly, the uniform lower bound on the systole is used in \cite{le2024quantum} in two steps. First, in the proof of spectral convergence, that is $N(X_n,I)+M(X_n,I)\sim \vol (X_n)$. Spectral convergence is enjoyed by Plancherel sequences, with no requirement of a lower bound on the systoles. The second step is the proof of the Quantum Ergodicity result for mean zero operators. We claim, without proving it now, that Plancherel convergence is enough for the result to hold. We will prove that Plancherel convergence is enough to prove Theorem \ref{QE theorem} for average zero operators in the following sections. The same argument can be used to prove that Plancherel convergence is enough to prove Theorem 1.2 in \cite{le2024quantum}.
\end{proof}

\subsection{Radial disintegration of integral operators}\label{radial disintegration section}

Let $k:\Hyp\times\Hyp\to\C$ be a measurable kernel with propagation bound $M>0$ and let $T:C(\Hyp)\to C(\Hyp)$ be the associated operator. For $0\le r\le M$ we define 
\begin{align*}
    T_r:C(\Hyp)\to& C(\Hyp) \\
    T_r f(x):=&\fint_{\partial B_r(x)} k(x,y)f(y)dA(y) \\
    =&\fint_{\mathbb{S}^1} k(x,\exp_x(r,\theta))f(\exp_x(r,\theta)) d\theta.
\end{align*}

If $\Gamma\le G$ is a lattice and $k$ is a $\Gamma$-invariant kernel, then $T_r$ descends to an operator on $C(X)$, where $X=\Gamma\setminus \Hyp$. In this case, for $s\in[0,\infty)$, we write 
\begin{align*}
\la T_r \ra_{s}:=&\fint_X\fint_{\partial B_r(x)} k(x,y)\phi_s(r) dA(y)dx \\ =&\phi_s(r)\fint_X\fint_{\mathbb{S}^1} k(x,\exp_x(r,\theta))d\theta dx.
\end{align*}
Further, if $k$ is a radial average zero kernel, we write 
$$\mathcal{D}_{X,I}(T_r):=\sum_{j\ \colon \lambda_j\in I}\Big |\la T_r\psi_j,\psi_j\ra\Big |,$$ 
$$\mathcal{C}_{X,I}(T_r):=\frac{1}{4\pi}\int_{\tau^{-1}(I)}\Big|\sum_{p=1}^q\la T_rE_p(s),E_p(s)\ra \Big|ds$$

and 

$$\operatorname{Dev}_{X,I}(T_r):=\frac{\mathcal{D}_{X,I}(T_r)+\mathcal{C}_{X,I}(T_r)}{N(X,I)+M(X,I)}$$

\begin{lemma}\label{radial disintegration}
Let $\Gamma\le G$ be a torsion free lattice and $X=\Gamma\setminus\Hyp$. If $k$ is a measurable $\Gamma$-invariant kernel with propagation bound $M>0$ and $I\subset\left(\frac{1}{4},\infty\right)$ is a compact interval, then
    \begin{equation}\label{disintegration of operator}
        T=\int_0^M T_r\sinh(r)dr
    \end{equation}
and for $s\in[0,\infty)$
    \begin{equation}\label{disintegration of correlation function}
        \la T\ra _s=\int_0^M \la T_r\ra _s\sinh(r) dr.
    \end{equation}

    Moreover, if $k$ has radial average zero, then 

    \begin{equation}\label{disintegration of quantum mean absolute deviation}
        \operatorname{Dev}_{X,I}(T)\le \sinh(M)\int_0^M\operatorname{Dev}_{X,I}(T_r) dr.
    \end{equation}
\end{lemma}

\begin{proof}
In order to prove (\ref{disintegration of operator}), fix $f\in C(\Hyp)$. Then 
    \begin{align*}
        \int_0^\infty T_rf(x)\sinh(r)dr&=\int_0^\infty\fint_{\mathbb{S}^1}k(x,\exp_x(r,\theta))f(\exp_x(r,\theta))\sinh(r)d\theta dr\\
        &=\int_\Hyp k(x,y)f(y)dy\\
        &=Tf(x).
    \end{align*}
To prove (\ref{disintegration of correlation function}), apply the argument above to $f=\phi_s$ and observe that 
$$\la T\ra_s=\fint_X T\phi_s(x) dx.$$
Now, to prove (\ref{disintegration of quantum mean absolute deviation}), observe that 
\begin{align*}
    \mathcal{D}_{X,I}(T)&=\sum_{j\ \colon \lambda_j\in I}\Big|\la T\psi_j,\psi_j\ra \Big|\\
    &=\sum_{j\ \colon \lambda_j\in I}\Big|\la \int_0^M T_r\psi_j\sinh(r)dr,\psi_j\ra \Big|\\
    &\le\sinh(M)\int_0^M\sum_{j\ \colon \lambda_j\in I}\Big|\la T_r\psi_j,\psi_j\ra \Big|dr\\
    &=\sinh(M)\int_0^M\mathcal{D}_{X,I}(T_r) dr.
\end{align*}
and similarly 
$$\mathcal{C}_{X,I}(T)\le\sinh(M)\int_0^M\mathcal{C}_{X,I}(T_r)dr.$$
\end{proof}

% \begin{lemma}
%     Let $k$ be a smooth kernel with compact support and radial everage zero. Then $[T_r]=0$ for every $r\ge 0$. 
% \end{lemma}
% \begin{proof}
%     \textcolor{red}{To do.}
% \end{proof}

\begin{rem}
    Lemma \ref{radial disintegration}, together with Proposition \ref{average zero proposition}, implies that Theorem \ref{QE theorem} is an immediate corollary of the following statement.
\end{rem}

\begin{propx}\label{prop: first big reduction}
    Let $I\subset \left(\frac{1}{4},\infty\right)$ be a compact interval and $(X_n)_{n\in\N}$ a sequence of finite area hyperbolic surfaces such that 
    \begin{enumerate}
        \item $(X_n)_{n\in\N}$ is a Plancherel sequence;
        \item there is a uniform spectral gap for the Laplacian on $X_n$;
        \item $\frac{q_n^2}{\vol(X_n)}\xrightarrow[]{n\to\infty}0$.
    \end{enumerate}
    Then, for any uniformly bounded sequence of measurable kernels $(k_n)_{n\in\N}$ with uniformly relative compact support, uniformly finite propagation, radial average zero and for any $r\ge0$, 
    $$\operatorname{Dev}_{X_n,I}((T_n)_r)\xrightarrow[]{n\to\infty}0.$$
\end{propx}

The proof of Proposition \ref{prop: first big reduction} is the content of the following two sections.

\subsection{Spectral bound}\label{spectral bound section}

Fix a compact interval $I\subset \left(\frac{1}{4},\infty\right)$, a finite-area hyperbolic surface $X$ with $q$ cusps and a radial average zero $\Gamma$-invariant measurable kernel $k\in L^\infty(\Hyp\times\Hyp)$ with finite propagation and with relative compact support in $X(\upsilon)$ for some $\upsilon\ge 1$. For $r\ge0$, let $T_r$ be the operator with singular kernel supported on $\{(x,y)\in X\times X\ \colon d(x,y)=r\}$ defined as in section \ref{radial disintegration section}.

For $t\ge0$, we introduce the wave propagation operator
$$P_tf(x):=e^{-t/2}\int_{B_t(x)} f(y) dy.$$
This is the operator by convolution with the radial kernel 
$$\rho_t:=\frac{\1_{[0,t]}}{e^{t/2}}.$$

In order to prove Proposition \ref{prop: first big reduction}, we want to formally replace the operator $T_r$ in $\operatorname{Dev}_{X,I}(T_r)$ with the time evolution operator $F_{T,r}:=\frac{1}{T}\int_0^T P_t T_r P_t dt$. Then, we reduce the estimate of the quantum mean absolute deviation $\operatorname{Dev}_{X,I}(F_{T,r})$ to an estimate of the Hilbert--Schmidt norm of $F_{T,r}$.
% The goal of this section is to show that this formal replacement is meaningful (see Lemma \ref{spectral bound lemma}). In order to do so, we first show that $F_{T,r}$ 

\begin{lemma}\label{lem: ker of F_T}
    $F_{T,r}$ is an operator by convolution with a bounded measurable kernel with finite propagation and relative compact support in $X(\upsilon+2T+r)$.
\end{lemma}
\begin{proof}
    First of all, notice that the operator $P_tT_rP_t$ has kernel 
$$k_{t,r}(x,y)=\frac{e^{-t}}{2\pi\sinh(r)}\int _{B_t(x)}\int_{\partial B_r(w)\cap B_t(y)}k(w,z)dA(z)dw.$$
In fact, 
\begin{align*}
    P_tT_rP_t\ f(x)&=e^{-t/2}\int _X \1_{[0,t]}(x,w) T_rP_t\ f(w)dw\\
    &=e^{-t/2}\int _X \1_{[0,t]}(x,w)\fint_{\partial B_r(w)} k(w,z) P_tf(z) dA(z)dw\\
    &=e^{-t}\int _X \1_{[0,t]}(x,w)\fint_{\partial B_r(w)} k(w,z)\int_X\1_{[0,t]}(y,z) f(y) dy dA(z)dw\\
    &=e^{-t}\int _X \int_X\1_{[0,t]}(x,w)\fint_{\partial B_r(w)}k(w,z)\1_{[0,t]}(y,z) dA(z)dwf(y)dy\\
    &=\frac{e^{-t}}{2\pi\sinh(r)}\int _X \int_{B_t(x)}\int_{\partial B_r(w)\cap B_t(y)} k(w,z) dA(z)dw f(y)dy.
\end{align*}

Hence, $k_{t,r}(x,y)=0$ whenever $d(x,y)\ge 2t+r$ and $||k_{t,r}||_\infty\le ||k||_\infty$, so that $k_{t,r}$ is a bounded measurable kernel with finite propagation. 

Now observe that the kernel of $F_{T,r}$ is 
$$K_{T,r}=\frac{1}{T}\int_0^T k_{t,r}dt.$$

Therefore $K_{T,r}$ has propagation bound $2T+r$. We now turn our attention to the support of $K_{T,r}$. By assumption, $k$ has relative compact support in $X(\upsilon)$. Hence, by construction, $K_{T,r}$ has relative compact support in $X(\upsilon+2T+r)$.

\end{proof}

\begin{rem}\label{rem: autom ker of F_T}
    The operator $F_{T,r}$, at first defined over $\Hyp$, descends to an integral operator $\textbf{F}_{T,r}$ on $X$ with automorphic kernel 
    $$\textbf{K}_{T,r}(x,y)=\sum_{\gamma\in\Gamma} K_{T,r}(x,\gamma.y).$$
    Notice that for each $x,y\in X$ the sum is finite, as $K_{T,r}$ has finite propagation. Moreover, As a consequence of Lemma \ref{lem: ker of F_T}, the kernel $\textbf{K}_{T,r}$ is supported in $X(\upsilon+2T+r)\times X(\upsilon+2T+r).$
\end{rem}

Recall that there exist constants $T_I>0$ and $C_I>0$ such that for any $T>T_I$ and $r\in \tau^{-1}(I)$
\begin{equation}\label{eq: bound on time evolved spherical transform}
    \frac{1}{T}\int_0^T \Big|\widehat{\rho_t}(r)\Big|^2dt\ge C_I.
\end{equation}
This is the content of \cite[Proposition 4.2]{le2017quantum}.  In the following we use the notation $a\lesssim b$ when there exists a universal constant $C>0$ such that $a\le Cb$. Also, we write $a\lesssim_I b$ if there is a constant $C_I$ depending only on $I$ such that $a\le C_I b$.

We have the following result, which is the analogous of Proposition 3.3 in \cite{le2024quantum}.

\begin{propx}\label{prop: spectral bound lemma}
There exist $T_I>0$ and $C(X,I,K_{T,r})>0$ such that for any $T>T_I$ 
$$\mathcal{D}_{X,I}(T_r)+\mathcal{C}_{X,I}(T_r)\lesssim_I C(X,I,K_{T,r})\Bigg|\Bigg|\textbf{F}_{T,r}\Bigg|\Bigg|_{HS},$$
where the constant $C(X,I,k)$ can be chosen to be
$$\max\big\{N(X,I),4q\log (\upsilon+2T+r)+q^2e^{-4\pi (\upsilon+2T+r)}+M(X,I)\}^{1/2}.$$
\end{propx}

We obtain Prop \ref{prop: spectral bound lemma} as a special case of the following more general result. 

\begin{lemma}\label{lem: HS bound}
    Let $F$ be an integral operator on $X$ with essentially bounded kernel $K:X\times X\to\C$  supported in $X(\upsilon)\times X(\upsilon)$ for some $\upsilon\ge 1$. Then 
    $$\mathcal{D}_{X,I}(F)+\mathcal{C}_{X,I}(F)\lesssim_I  C(X,I,K)\Big|\Big| F\Big|\Big|_{HS},$$
    where the constant $C(X,I,K)$ can be chosen to be 
    $$\max\big\{N(X,I),4q\log (\upsilon)+q^2e^{-4\pi \upsilon}+M(X,I)\}^{1/2}.$$
\end{lemma}

\begin{rem}
The strategy of the proof of Lemma \ref{lem: HS bound} is adapted from Proposition 3.3 of \cite{le2024quantum}. In adapting the argument, we noticed a gap in the proof. More precisely, the last inequality is claimed to follow directly from an application of the spectral theorem to the kernel of a certain operator. However, the argument using the spectral theorem in this way requires the kernel to be radial, whereas the kernel under consideration is not.

We remedy this issue as follows. First, we apply the spectral theorem not directly to the kernel itself, but to the functions obtained by fixing one variable of the kernel and varying the other. We then introduce a cutoff for the Eisenstein series, which is justified by the compact support of the kernel. Successive applications of the Cauchy--Schwarz inequality then allow us to separate the Hilbert--Schmidt norm of the operator from a factor involving the $L^2$-norm of the truncated Eisenstein series. The latter factor is controlled using the Maaß--Selberg relations, which provide the required quantitative estimate.

We thank Tuomas Sahlsten for a fruitful discussion concerning this proof, and in particular for pointing us to the quantitative estimate coming from the Maaß--Selberg relations.
\end{rem}

\begin{proof}
    As $K$ is by assumption (essentially) bounded, for (almost) every fixed $y\in X$ we have $K(\cdot,y)\in L^2(X)$. The spectral theorem for hyperbolic surfaces (\cite[Thm. 4.7+7.3]{iwaniec2021spectral}) then implies that 
    \begin{equation}\label{eq: norm of kernel}
        \sum_j\Big|\la K(\cdot,y),\psi_j\ra\Big|^2+\frac{1}{4\pi}\int_\R\sum_{p=1}^q\Big|\la K(\cdot,y),E_p(s)\ra\Big|^2ds=\big|\big|K(\cdot,y)\big|\big|_2^2.
    \end{equation}
    Observe that 
    $$\int_X\big|\big|K(\cdot,y)\big|\big|_2^2dy=\Big|\Big| F\Big|\Big|_{HS}^2.$$
    
    We now estimate $\mathcal{D}_{X,I}(F)$ and $\mathcal{C}_{X,I}(F)$ separately. Applying Cauchy--Schwarz first to $L^2(X)$ and then to $\C^{N(X,I)}$ yields 
    \begin{equation}\label{eq: discrete contribution}
        \begin{split}
            \mathcal{D}_{X,I}(F)\le&\sum_j\big|\big| F\psi_j\big|\big|\\
            \le &\sqrt{N(X,I)}\left(\sum_j \big|\big| F\psi_j\big|\big|^2\right)^{1/2}.
        \end{split}
    \end{equation}

    The continuous spectrum requires more attention. Let $\1_\upsilon$ be the characteristic function of $X(\upsilon)$. Then $\la FE_p(s),E_p(s)\ra=\la FE_p(s),\1_\upsilon E_p(s)\ra$. This follows from the kernel having support in $X(\upsilon)\times X(\upsilon)$. This cutoff argument allows us to apply Cauchy--Schwarz inequality first to $L^2(X)$, then to $L^2(\tau^{-1}(I)\times \mathcal{C}(X))$, and obtain 
    \begin{equation}\label{eq: continuous contribution}
        \begin{split}
           &\int_{\tau^{-1}(I)}\sum_p \Big|\la F E_p(s),E_p(s)\ra\Big| ds\le\\
           &\int_{\tau^{-1}(I)}\sum_p\big|\big| FE_p(s)\big|\big|\cdot \big|\big| \1_\upsilon E_p(s)\big|\big| ds\le\\
           &\left(\int_{\tau^{-1}(I)}\sum_p \big|\big| \1_\upsilon E_p(s)\big|\big|^2ds\right)^{\frac{1}{2}}\left(\int_{\tau^{-1}(I)}\sum_p \big|\big| FE_p(s)\big|\big|^2ds\right)^{\frac{1}{2}}.
        \end{split}
    \end{equation}

    Now notice that 
    $$\big|\big| F\psi_j\big|\big|^2=\int_X\big|\la K(\cdot,y),\psi_j\ra\big|^2dy$$
    and 
    $$\big|\big| FE_p(s)\big|\big|^2=\int_X\big|\la K(\cdot,y),E_p(s)\big|^2 dy.$$
    Hence, combining (\ref{eq: norm of kernel}), (\ref{eq: discrete contribution}) and (\ref{eq: continuous contribution}), together with the concavity of the square root ($\sqrt{a}+\sqrt{b}\le\sqrt{2}\sqrt{a+b}$), we obtain
    $$\mathcal{D}_{X,I}(F)+\mathcal{C}_{X,I}(F)\lesssim\max\left\{N(X,I),\int_{\tau^{-1}(I)}\sum_p \big|\big| \1_\upsilon E_p(s)\big|\big|^2ds\right\}^{1/2}\Big|\Big|F\Big|\Big|_{HS}.$$
    
    We conclude the proof by recalling that by means of Maaß-Selberg relations the following estimate holds (see \cite[section 4]{le2024quantum}):
    $$\sum_p||\1_\upsilon E_p(s)||^2\lesssim_I4q\log \upsilon+q^2e^{-4\pi \upsilon}+\frac{-\varphi_X'}{\varphi_X}\left(\frac{1}{2}+is\right).$$

    Integrating over $\tau^{-1}(I)$ yields 
    $$\int_{\tau^{-1}(I)}\sum_p||\1_\upsilon E_p(s)||^2ds\lesssim_I 4q\log \upsilon+q^2e^{-4\pi \upsilon}+M(X,I).$$
\end{proof}

\begin{proof}[proof of Proposition \ref{prop: spectral bound lemma}]
    Let $T>0$ be big enough so that (\ref{eq: bound on time evolved spherical transform}) holds. Using the fact that for every $t\ge0$ the operator $P_t$ is self-adjoint, one obtains 
    $$\Big|\la T_r\psi_j,\psi_j\ra\Big|\le\frac{1}{C_I}\Big|\la \textbf{F}_{T,r}\psi_j,\psi_j\ra\Big|,$$
    $$\Big|\la T_r E_p(s),E_p(s)\ra\Big|\le\frac{1}{C_I}\Big|\la\textbf{F}_{T,r} E_p(s),E_p(s)\ra\Big|.$$
    Therefore, 
    $$\mathcal{D}_{X,I}(T_r)+\mathcal{C}_{X,I}(T_r)\le\frac{1}{C_I}\bigg(\mathcal{D}_{X,I}(\textbf{F}_{T,r})+\mathcal{C}_{X,I}(\textbf{F}_{T,r})\bigg).$$
    The result now follows from Lemma \ref{lem: HS bound} applied to $\textbf{F}_{T,r}$.
\end{proof}

\subsection{Geometric bound}\label{geometric bound}

The only thing left in order to prove Proposition \ref{prop: first big reduction}, and therefore Theorem \ref{QE theorem}, is to provide a suitable bound for $\big|\big| \textbf{F}_{T,r}\big|\big|_{HS}$. The idea is to split the integral defining the Hilbert-Schmidt norm of $\textbf{F}_{T,r}$ in two parts, over points with small and large injectivity radius respectively, in such a way that the thin part can be controlled by Plancherel convergence and the thick part can be controlled by means of an ergodic theory result for averaging operators.
\begin{propx}\label{prop: geometric bound}
    Let $k\in L^\infty(\Hyp\times\Hyp)$ be a $\Gamma$-invariant kernel with propagation bound $R$ and relative compact support in $X(\upsilon)$ for some $\upsilon\ge 1$. For $r\ge 0$, let $T_r$ be the operator $T_r f(x):=\fint_{\partial B_r(x)} k(x,y) f(y) dA(y)$. Then for every $T\ge 0$ the operator $\textbf{F}_{T,r}$ is Hilbert--Schmidt and 
    \begin{align*}
        \Big|\Big| \textbf{F}_{T,r}\Big|\Big|_{HS}^2\lesssim& \frac{||k||_\infty^2}{T \rho(\lambda_1)}\operatorname{vol(X)}\\
        &+ e^{4T+2r}||k||_\infty^2\big(\vol X_{\le 2T+r}+\int_{X(\upsilon+2T+r)}\sharp\big(\Gamma^\star.y\cap \overline{B_{6T+3r}}(y)\Big)dy,
    \end{align*}
    where $\rho(\lambda_1)$ only depends on the spectral gap $\lambda_1$ of $X$.
\end{propx}

We obtain Proposition \ref{prop: geometric bound} through a finer version of Lemma 3.5 in \cite{le2024quantum}.

\begin{lemma}\label{lem: geometric bound}
    Let $K\in L^\infty(\Hyp\times\Hyp)$ be a $\Gamma$-invariant kernel with propagation bound $R>0$ and relative compact support in $X(\upsilon)$ for some $\upsilon\ge 1$. If $F$ is the integral operator on $X$ with kernel $\sum_{\gamma\in\Gamma} K(x,\gamma y)$, then
    \begin{equation}\label{eq: geometric bound}
    \begin{split}
        \big|\big| F\big|\big|_{HS}^2\le& 2\int_\mathcal{F}\int_\Hyp |K(x,y)|^2 dxdy\\
        &+e^{2R}||K||_{\infty}^2\Big(\operatorname{vol} X_{\le R}+\int_{X(\upsilon)}\sharp\big( \Gamma^\star. y\cap \overline{B_{3R}}(y)dy\Big),
    \end{split}
    \end{equation}
    where $\Gamma^\star:=\Gamma\smallsetminus\{\operatorname{id}\}$ and $\mathcal{F}$ is a fundamental domain for $X$ in $\Hyp$.
\end{lemma}

\begin{proof}
    The Hilbert-Schmidt norm of $F$ is 
    $$\big|\big| F\big|\big|_{HS}^2=\int_X\int_X \Big|\sum_{\gamma\in\Gamma} K(x,\gamma. y)\Big|^2 dxdy.$$

We split the integral in the variable $y$ in two parts: over the points with small and large injectivity radius:
\begin{equation}\label{eq: first HS reduction}
\begin{split}
    \big|\big| F\big|\big|_{HS}^2=&\int_{X_{>R}}\int_X\Big|\sum_{\gamma\in\Gamma} K(x,\gamma .y)\Big|^2 dxdy\\
    &+\int_{X_{\le R}}\int_X\Big|\sum_{\gamma\in\Gamma} K(x,\gamma. y)\Big|^2 dxdy.
\end{split}
\end{equation}
First, notice that the sum in the first integral reduces to one term. In fact, let $y\in X_{>R}$ and $x\in X$ be such that $d(x,y)\le R$. Then $d(x,\gamma .y)>R$ for every $\gamma\in\Gamma\smallsetminus\{\operatorname{id}\}$, for otherwise
$$d(y,\gamma. y)\le d(y,x)+d(x,\gamma .y)\le 2R,$$
which is a contradiction with $y\in X_{>R}$. Hence, the first summand in (\ref{eq: first HS reduction}) is 
$$\int_{X_{>R}}\int_X\sum_{\gamma\in\Gamma}|K(x,\gamma .y)|^2dxdy\le\int_\Hyp\int_X |K(x,y)|^2 dxdy.$$
Now, by splitting the integral in the variable $x$ as done previously for $y$, the second summand in (\ref{eq: first HS reduction})becomes 
$$\int_{X_{\le R}}\int_{X_{>R}}\Big|\sum_{\gamma\in\Gamma} K(x,\gamma. y)\Big|^2 dxdy+\int_{X_{\le R}}\int_{X_{\le R}}\Big|\sum_{\gamma\in\Gamma} K(x,\gamma .y)\Big|^2 dxdy.$$

The sum in the first integral reduces to one term for the same reason discussed above, so that 
\begin{equation}\label{eq: second HS reduction}
    \Big|\Big| F\Big|\Big|_{HS}^2\le 2\int_\mathcal{F}\int_\Hyp |K(x,y)|^2 dxdy+\int_{X_{\le R}}\int_{X_{\le R}}\Big|\sum_{\gamma\in\Gamma} K(x,\gamma. y)\Big|^2dxdy.
\end{equation}

We have left to bound the second summand in the RHS of (\ref{eq: second HS reduction}). First of all, we notice that the integrals are supported in $X(\upsilon)\cap X_{\le R}$, due to the kernel $K$ having relative compact support in $X(\upsilon)$. Also, Cauchy--Schwarz inequality yields 
$$\Big|\sum_{\gamma\in\Gamma} K(x,\gamma .y)\Big|^2\le \sharp\Big(\Gamma.y\cap \overline{B_R}(x)\Big)\Big(\sum_{\gamma\in\Gamma} |K(x,\gamma .y)|^2\Big).$$
Now observe that for $x\in X_{\le R}$ and $y\in X$
$$\sharp\Big(\Gamma. y\cap\overline{B_R}(x)\Big)\le\sharp\big(\Gamma.y\cap\overline{B_{3R}}(y)\Big).$$
In fact, there exists $\gamma_x\in\Gamma\smallsetminus\{\operatorname{id}\}$ such that $d(x,\gamma_x. x)\le R$, so that if $\gamma\in\Gamma$ is such that $d(x,\gamma. y)\le R$ we have 
$$d(y,\gamma^{-1}\gamma_x \gamma. y)\le 3R.$$
Therefore,
$$\int_{X_{\le R}}\int_{ X_{\le R}} \Big|\sum_{\gamma\in\Gamma} K(x,\gamma. y)\big|^2 dxdy\le$$
$$\int_{ X_{\le R}\cap X(Y)}\sharp\big(\Gamma. y\cap \overline{B_{3R}}(y)\Big)\left(\int_X\sum_{\gamma\in\Gamma}|K(x,\gamma .y)|^2dx\right)dy.$$

The inner integral has the following bound: 
$$\int_X\sum_{\gamma\in\Gamma}|K(x,\gamma. y)|^2dx\le\int_\Hyp |K(x,y)|^2 dx=\int_{B_R(y)}|K(x,y)|^2 dx\le e^{2R}||K||_\infty^2,$$
where $e^{2R}$ bounds the area of a hyperbolic ball of radius $R$. 
Summarizing, 
$$\Big|\big|F\Big|\Big|_{HS}^2\le 2\int_X\int_\Hyp |K(x,y)|^2 dxdy+e^{2R}||K||_\infty^2\int_{X_{\le R}\cap X(\upsilon)}\sharp\Big(\Gamma.y\cap\overline{B_{3R}}(y)\Big) dy.$$
We conclude the proof by observing that the contribution of the identity element in the last integral is controlled by $\operatorname{vol} X_{\le R}$ and the contribution of all non-identity elements is controlled by 
$$\int_{X(\upsilon)}\sharp\Big(\Gamma^\star.y\cap \overline{B_{3R}}(y)\Big)dy.$$
    
\end{proof}

In order to proof Proposition \ref{prop: geometric bound} we want to apply Lemma \ref{lem: geometric bound} to the kernel $K_{T,r}$. We also use the following result. 

\begin{propx}\label{prop: application of Nevo}
     $$\int_\mathcal{F}\int_\Hyp |K_{T,r}(x,y)|^2 dxdy\lesssim\frac{||k||_\infty}{T\rho(\lambda_1)}\vol X,$$
    where $\theta(\lambda_1)$ is a positive constant only depending on the spectral gap of the Laplacian.
\end{propx}

The above is analogous to Lemma 3.6 in \cite{le2024quantum}, where the role of the test function is taken by the distributional kernel of $T_r$ supported on $d(x,y)=r$. The special case of multiplication operators is retrieved when $r=0$ and the assumption in \cite{le2024quantum} that $\la a\ra=0$ is here replaced by the assumption $[k](r)=0$. We show how to adapt the proof to our case in section \ref{ergodic theorem and averaging operators}.

\begin{proof}[proof of Proposition \ref{prop: geometric bound}]
    We recall from Lemma \ref{lem: ker of F_T} that the automorphic kernel $\textbf{K}_{T,r}(x,y)=\sum_{\gamma\in\Gamma} K_{T,r}(x,\gamma.y)$ of $\textbf{F}_{T,r}$ is such that $K_{T,r}$ has propagation bound $2T+r$ and relative compact support in $X(\upsilon+2T+r)$. Therefore, applying Lemma \ref{lem: geometric bound} to $\textbf{F}_{T,r}$ we get 
    \begin{align*}
        \Big|\Big| \textbf{F}_{T,r}\Big|\Big|_{HS}^2\le&2\int_\mathcal{F}\int_\Hyp |K_{T,r}(x,y)|^2 dxdy\\
        &+e^{2(2t+r)}||K_{T,r}||_\infty^2\big(\vol X_{\le 2T+r}+\int_{X(\upsilon+2T+r)}\sharp\big(\Gamma^\star.y\cap\overline{B_{3(2T+r)}}(y)\Big)dy.
    \end{align*}
    
    We conclude by observing that $||K_{T,r}||_\infty\le ||k||_\infty$ and by applying Proposition \ref{prop: application of Nevo} to estimate the first integral.

\end{proof}

\begin{proof}[proof of Proposition \ref{prop: first big reduction}]
        Let $(k_n)_{n\in\N}$ be a sequence of measurable kernels with uniformly relative compact support, uniformly finite propagation and radial average zero. Firstly, recall that Plancherel sequences enjoy the spectral convergence property $N(X_n,I)+M(X_n,I)\approx \vol (X_n)$ as $n\to\infty$ (this is the content of Prop. \ref{prop: spectral convergence}). Then, Proposition \ref{prop: spectral bound lemma} implies that for $T$ big enough
        $$\operatorname{Dev}_{X_n,I}((T_n)_r)^2\le \frac{C(X_n,I,K_{T,r,n})}{N(X_n,I)+M(X_n,I)}\frac{\big|\big|\textbf{F}_{T,r,n}\big|\big|_{HS}^2}{N(X_n,I)+M(X_n,I)}.$$
         
        For the first factor, notice that the condition $\frac{q_n^2}{\vol X_n}\to 0$ implies that 
        $$\lim_{n\to\infty}\frac{ C(X_n,I, K_{T,r,n})}{N(X_n,I)+M(X_n,I)}\le 1.$$

        For the second factor, Let $L$ be a uniform upper bound on the $L^\infty$ norm of $k_n$, and $\beta$ be a uniform lower bound on the spectral gap of $X_n$. These exist by assumption. Also, let $\upsilon>0$ be big enough so that $k_n$ has relative compact support in $X_n(\upsilon)$ for every $n$. Then Proposition \ref{prop: geometric bound} implies that 
        \begin{align*}
            \frac{\big|\big| \textbf{F}_{T,r}\big|\big|_{HS}^2}{\vol X_n}&\lesssim \frac{L^2}{T\rho(\beta)}+\\
            &L^2e^{4T+2r}\Bigg(\frac{\vol (X_n)_{\le 2T+r}}{\vol X_n}+\frac{1}{\vol X_n}\int_{X(\upsilon+2T+r)}\sharp\Big(\Gamma^\star_n.y\cap\overline{B_{6T+3r}}(y)\Big)dy\Bigg).
        \end{align*}
        Now, $\frac{\vol (X_n)_{\le 2T+r}}{\vol X_n}\xrightarrow[]{n\to\infty}0$ because Plancherel sequence are Benjamini--Schramm convergent to $\Hyp$. Moreover, Lemma \ref{lem: Plancherel and parabolic contribution} implies that 
        $$\frac{1}{\vol X_n}\int_{X(\upsilon+2T+r)}\sharp\Big(\Gamma^\star_n.y\cap\overline{B_{6T+3r}}(y)\Big)dy\xrightarrow[]{n\to\infty}0.$$
        Summarizing, for every $T$ big enough 
        $$\lim_{n\to\infty}\operatorname{Dev}_{X_n,I}(T_r)\lesssim \frac{1}{T}.$$
        Letting $T\to\infty$ concludes the proof.
 \end{proof}

\subsubsection{Quantitative ergodic estimate}\label{ergodic theorem and averaging operators}

This section is dedicated to the proof of Proposition \ref{prop: application of Nevo}. We adapt the technique developed in \cite{le2017quantum}. The main tool is an ergodic Theorem  about the equidistribution of averaging operators applied to mean zero functions in some probability space. Before stating the Theorem we need some terminology. Let $G$ be a locally compact group. Fix a Haar measure on $G$ and assume that $G$ admits a measure-preserving action $\pi$ on a probability space $(Y,\nu)$. Then the action of $G$ defines a representation on $L^2(Y)$ by $\pi(g)f(x):=f(\pi(g^{-1})x)$. Given a measurable set $L\subset G$ of finite volume we define an averaging operator on $L^2(Y,\nu)$ by
$$\pi (L)f(x):=\frac{1}{|L|}\int_L f(\pi(g^{-1})x)dg.$$
The \textit{matrix coefficients} of the representation of the representation are 
$$C_{f,h}(g):=\la \pi(g) f,h\ra_{L^2(Y,\nu)}$$
for $g\in G$ and $f,h\in L^2_0(Y):=\{f\in L^2(Y) \ \colon \int_Y fd\nu=0\}$. The \textit{integrability exponent} of $L_{|L^2_0(Y)}$ is
$$q_0:=\inf\{ q>0\ \colon C_{f,h}\in L^p(G) \text{ for all } f,h \text{ in a dense subset of }L^2_0(Y)\}.$$
The action of $G$ on $Y$ has a \textit{spectral gap} if $q_0<\infty$.

\begin{thm}[Nevo's ergodic Theorem, \cite{gorodnik2015quantitative} 4.3]\label{Nevo's ergodic theorem}
    Let $G$ be a simple Lie group with a measure preserving action on a probability space $(Y,\nu)$ with a spectral gap. Then there exist $\theta,C>0$ such that for $t\ge 0$ and a family $L_t\subset G$ of measurable sets of finite measure we have 
    $$\Big|\Big| \pi(L_t) f-\int_Yf d\nu\Big|\Big|_{L^2(Y,\nu)}\le \frac{C}{|L_t|^\theta} ||f||_{L^2(Y,\nu)}$$
    for any $f\in L^2(Y,\nu)$. The constant $C$ depends only on the group $G$ and $\theta$ depends only on the integrability exponent.
\end{thm}

We will apply Theorem \ref{Nevo's ergodic theorem} to $G=PSL_2(\R)$ and $Y= X\times\mathbb{S}^1$, where $X=\Gamma\setminus \Hyp$ is the hyperbolic surface we have fixed since section \ref{spectral bound section}. $\PSL_2(\R)$ acts on $Y$ by $\pi(g) \Gamma x:=\Gamma xg^{-1}$, so that the induced representation on $L^2(Y )$ is right multiplication
$$\pi(g) f(\Gamma x)=f(\pi(g^{-1})\Gamma x)=f(\Gamma xg).$$
It is a known fact that the integrability exponent for the action of $\PSL_2(\R)$  only depends on the spectral gap of the Laplacian on $\Gamma\setminus \Hyp$ (see \cite[Chapter V, Proposition 3.1.5]{howe2012non}).

\begin{proof}[proof of Proposition \ref{prop: application of Nevo}]
    Let $\mathcal{F}$ be a fundamental domain of $X$ in $\Hyp$. Our goal is to bound 
    $$\int_\mathcal{F}\int_\Hyp |K_{T,r}(x,y)|^2 dxdy=$$
    $$\int_\mathcal{F}\int_\Hyp\Big|\frac{1}{T}\int_0^T \frac{e^{-t}}{2\pi\sinh r}\int_{B_t(x)}\int_{\partial B_r(w)\cap B_t(y)}k(w,z)dA(z)dwdt\Big|^2dxdy.$$

By writing the integral in the variable $x$ in polar coordinates, using invariance of the Liouville measure under the geodesic flow (or using a change of variable similar to the one introduced in \cite[Lemma 7.1]{le2017quantum}) and recalling that the kernel $K_{T,r}$ has propagation bound $2T+r$ we can rewrite the above integral as
%$$\int_0^{2T+r}\sinh l\int_{\mathbb{S}^1}\int_\mathcal{F}\Big|\frac{1}{T}\int_{l/2}^T\frac{e^{-t}}{2\pi\sinh r}\int_{B_t\Big(\exp_z\big(r-\tfrac{l}{2},\theta\big)\Big)\cap B_t(\exp_z\big(r+\tfrac{l}{2},\theta\big)\Big)}\int_{\partial B_r(w)\cap B_t\Big(\exp_z\big(r+\tfrac{l}{2},\theta\big)\Big)} k(w,z)dA(z)dw\Big|^2dzd\theta dl$$

$$\int_0^{2T+r}\sinh l\int_{\mathbb{S}^1}\int_\mathcal{F}\Big|\frac{1}{T}\int_{\tfrac{l-r}{2}}^T\frac{e^{-t}}{2\pi\sinh r}\int_{B_t(p^-_z)\cap B_{t+r}(p^+_z)}\int_{\partial B_r(w)\cap B_t(p^+_z)} k(w,v)dA(v)dwdt\Big|^2dzd\theta dl,$$

where we set 
$$p^-_z:=\exp_z\big(\tfrac{l-r}{2},\theta\big)\quad \text{and}\quad  p^+_z:=\exp_z\big(\tfrac{l+r}{2},\theta\big).$$

\begin{figure}[htp]
    \centering
    \includegraphics[width=1\linewidth]{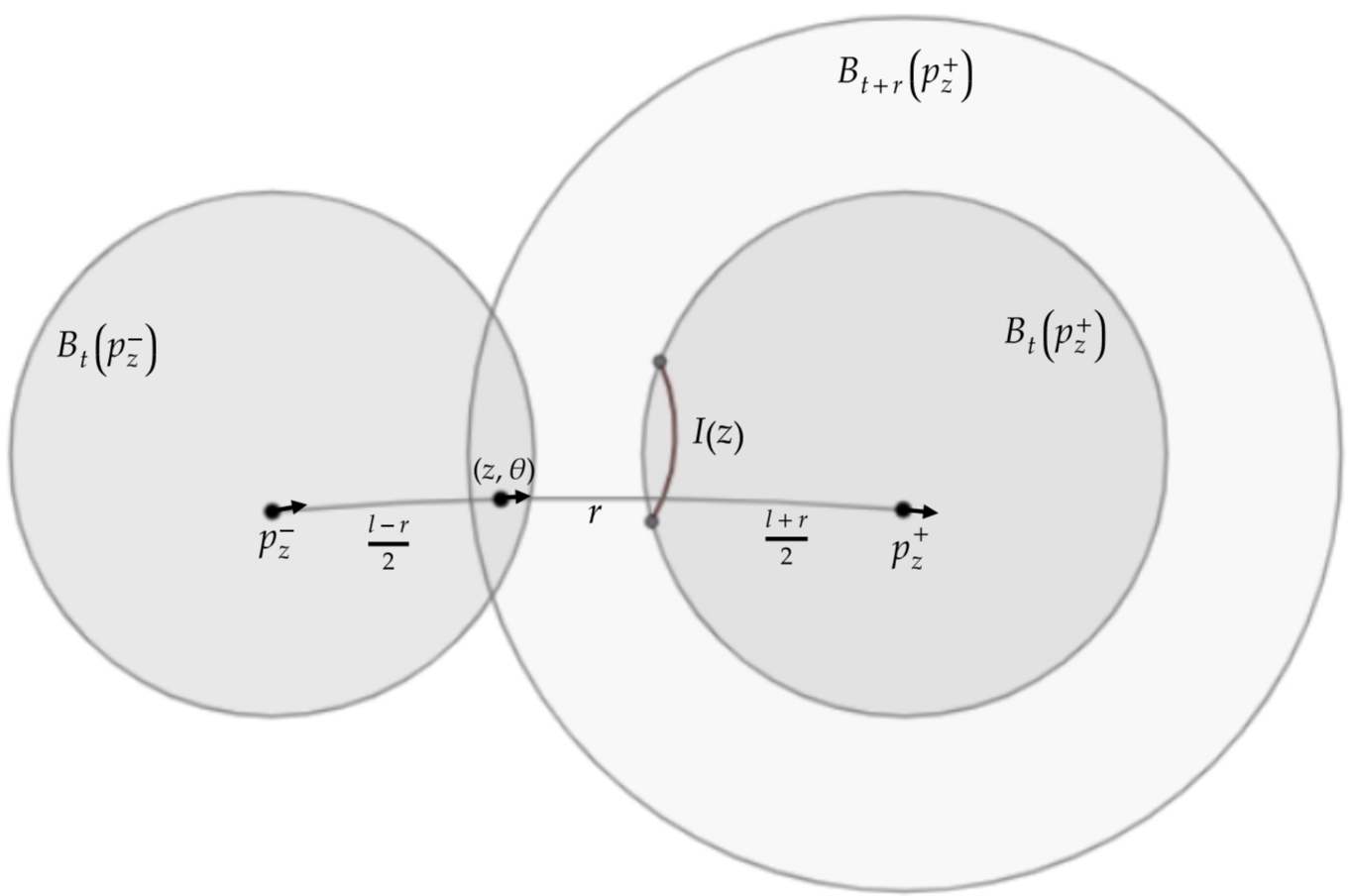}
%    \caption{add caption}
%   \label{figure 1}
\end{figure}

Consider the map $a_r:\mathcal{F}\times\mathbb{S}^1\to\C$, $a_r(z,\theta):=k(z,\exp_z(r,\theta))$. Now we want to write 
    $$\int_{B_t(p^-_z)\cap B_t(p^+_z)}\int_{\partial B_r(w)\cap B_t(p^+_z)} k(w,v)dA(v)dw=|L_t(l,r)|\pi(L_t(l,r))a_r(z,\theta)$$
for a suitable choice of measurable sets $L_t(l,r)\subset PSL_2(\R)$. This will allow us to apply Nevo's ergodic Theorem. For $w\in B_t(x_z)\cap B_t(y_z)$ we define 
$$I(w):=\{\eta\in\mathbb{S}^1\ \colon \exp_w(r,\eta)\in B_t(y_z)\}.$$

 We can choose $L_t(l,r)$ to be the set of elements in $\PSL_2(R)$ such that for any $(z,\theta)\in\Hyp\times\mathbb{S}^1$ the set $\pi(L_t(l,r)^{-1})(z,\theta)$ is 
 $$\bigcup_{w\in B_t(p^-_z)\cap B_t(p^+_z)}w\times I_w\subset \Hyp\times \mathbb{S}^1.$$

% A couple of word on why the choice of $L_t(l,r)$ does not depend on the base point $(z,\theta)$. It all boils down to the group $\PSL_2(\R)$ acting by isometries on the unit tangent bundle of $\Hyp$. In fact it is enough to define $L_t(l,r)_o$ for a base point, say $o=(i,\uparrow)$. Then notice that any point $(z,\theta)$ corresponds to a unique element $g\in G$ such that $(z,\theta)=g.o$. Then the lens space around $(z,\theta)$ is the lens space around $o$ translated by $g$. 

% In this section we show how the Hilbert-Schmidt norm of the propagator is controlled by the asymptotic dynamics of the geodesic flow.

Summarizing, 

$$\int_\mathcal{F}\int_\Hyp |K_{T,r}(x,y)|^2 dxdy\lesssim$$
$$\int_0^{2T+r}\sinh l\int_{\mathbb{S}^1}\int_\mathcal{F}\Big|\frac{1}{T}\int_{\tfrac{l-r}{2}}^T\frac{|L_t(l,r)|}{e^t}\pi(L_t(l,r))a_r(z,\theta)dt\Big|^2dzd\theta dl\le$$
$$\int_0^{2T+r}\sinh l\Big(\frac{1}{T}\int_{\tfrac{l-r}{2}}^T \frac{|L_t(l,r)|}{e^t}\big|\big|\pi(L_t(l,r)a_r\big|\big|_{L^2(X\times\mathbb{S}^1)}dt\Big)^2dl,$$
where the last line is obtained applying Minkowski integral inequality. Now we observe that $\fint_X\fint_{\mathbb{S}^1} a_r(z,\theta)d\theta dz=[k](r)=0$, so that by Theorem \ref{Nevo's ergodic theorem}
$$\big|\big|\pi(L_t(l,r))a_r\big|\big|_2\le\frac{C}{|L_t(l,r)|^{\theta(\lambda_1)}}||a_r||_2,$$
where we highlight that the coefficient $\theta(\lambda_1)$ only depends on the spectral gap $\lambda_1$.

Therefore,
$$\int_\mathcal{F}\int_\Hyp |K_{T,r}(x,y)|^2 dxdy\lesssim\frac{||a_r||_2^2}{T^2}\int_0^{2T+r}\sinh l\Big(\int_{\tfrac{l-r}{2}}^T\frac{|L_t(l,r)|^{1-\theta(\lambda_1)}}{e^t}dt\Big)^2dl.$$

Now observe that $||a_r||_2\le 2\pi\vol X||k||_\infty$ and that for any $(z,\theta)\in\Hyp\times\mathbb{S}^1$
$$|L_t(l,r)|\le |B_{t+r}(x_z)\cap B_{t+r}(y_z)|=O(e^{t+r-l/2}).$$
Therefore, for every $l\in[0,2T+r]$, 
$$\sinh l\Big(\int_{\tfrac{l-r}{2}}^T\frac{|L_t(l,r)|^{1-\theta(\lambda_1)}}{e^t}dt\Big)^2\lesssim\frac{1}{\theta(\lambda_1)^2}.$$
Hence, by writing $\rho(\lambda_1):=\theta(\lambda_1)^2$, we conclude 
$$\int_\mathcal{F}\int_\Hyp |K_{T,r}(x,y)|^2 dxdy\lesssim\frac{||k||_\infty^2}{T\rho(\lambda_1)}\vol X.$$

\end{proof}

\bibliographystyle{abbrv}
{\small\bibliography{ref}}

\end{document}